\DeclareSymbolFont{AMSb}{U}{msb}{m}{n}
\documentclass[11pt,a4paper,reqno,noamsfonts]{amsart}
\makeatletter 
\newcommand{\mylabel}[2]{#2\def\@currentlabel{#2}\label{#1}}
\renewcommand\@biblabel[1]{#1.}
\makeatother
\usepackage[english]{babel}
\usepackage[dvipsnames]{xcolor}
\usepackage{graphicx,pifont} 
\usepackage[utopia]{mathdesign}
\usepackage{bbm}
\usepackage{mathtools}
\usepackage[scr,scaled=1.0]{rsfso}
\usepackage[bb=boondox]{mathalfa}
\usepackage{tikz-cd}
\usepackage{stmaryrd}
\usepackage{tikz}
\usepackage{tkz-tab}
\usetikzlibrary{arrows,automata,decorations.markings,cd,arrows.meta}
\usetikzlibrary{calc}
\usetikzlibrary{positioning}
\usepackage{extarrows}

 \usepackage[pdftex,
                paper=a4paper,
                portrait=true,
                textwidth=160mm,
                textheight=247mm,
                tmargin=2.5cm,
                marginratio=1:1]{geometry}

\usepackage[utf8]{inputenc}
\usepackage{braket,caption,comment,mathtools,stmaryrd}
\usepackage[usestackEOL]{stackengine}
\usepackage{multirow,booktabs,microtype,relsize}
\usepackage[colorlinks,bookmarks]{hyperref} %
      \hypersetup{colorlinks,%
            citecolor=olive,%
            filecolor=black,%
            linkcolor=teal,%
            urlcolor=green}
      \numberwithin{equation}{section}
\usepackage[capitalise]{cleveref}

\newtheorem{theorem}{Theorem}[section]
\newtheorem{proposition}[theorem]{Proposition}
\newtheorem{lemma}[theorem]{Lemma}

\newtheorem{corollary}[theorem]{Corollary}

\theoremstyle{definition}

\newtheorem{example}[theorem]{Example}

\theoremstyle{remark}

\renewcommand{\AA}{\ensuremath{\mathbbmss{A}}}

\newcommand{\GG}{\ensuremath{\mathbbmss{G}}}
\newcommand{\PP}{\ensuremath{\mathbbmss{P}}}
\newcommand{\ZZ}{\ensuremath{\mathbbmss{Z}}}

\newcommand{\Ocal}{\ensuremath{\mathscr{O}}}

\newcommand{\uss}{\ensuremath{\mathrm{ss}}}

\newcommand{\gfrak}{\ensuremath{\mathfrak{g}}}
\newcommand{\mfrak}{\ensuremath{\mathfrak{m}}}
\newcommand{\pfrak}{\ensuremath{\mathfrak{p}}}
\newcommand{\qfrak}{\ensuremath{\mathfrak{q}}}

\def\uev{{\mathrm{ev}}}

\DeclareMathOperator{\id}{id}
\DeclareMathOperator{\ev}{ev}
\DeclareMathOperator{\Hom}{Hom}

\DeclareMathOperator{\Spec}{Spec}
\DeclareMathOperator{\sSpec}{\mathbf{Spec}}
\DeclareMathOperator{\Proj}{Proj}
\DeclareMathOperator{\sProj}{\mathbf{Proj}}
\DeclareMathOperator{\Sets}{Sets}
\DeclareMathOperator{\sVect}{sVect}
\DeclareMathOperator{\sAlg}{sAlg}
\DeclareMathOperator{\sSch}{sSch}
\DeclareMathOperator{\Groups}{Groups}

\DeclareMathOperator{\Lie}{Lie}
\DeclareMathOperator{\ad}{ad}
\DeclareMathOperator{\Sym}{Sym}
\DeclareMathOperator{\tot}{tot}

\newcommand{\lbbar}{\{\kern-0.76ex\{}
\newcommand{\rbbar}{\}\kern-0.76ex\}}

\newcommand{\medwedge}{\mathbin{\scalebox{1.2}{$\wedge$}}}

\newcommand*{\sbullet}{\raisebox{0.1ex}{\scalebox{0.6}{$\bullet$}}}

\title{Geometric Invariant Theory for Affine Superschemes}

\author{Alexander Quintero V\'{e}lez} 
\address{Departamento de Matem\'{a}ticas\\ Universidad Nacional de Colombia Sede Medell\'{i}n \\ Carrera 65 $\#$ 59A--110 \\ Medell\'{i}n \\ Colombia}
\email{aquinte2@unal.edu.co}

\author{Pedro Rizzo} 
\address{Instituto de Matem\'{a}ticas \\ Universidad de Antioquia \\ Calle 62 $\#$ 52-59 \\ Medell\'{i}n \\ Colombia}
\email{pedro.hernandez@udea.edu.co}

\author{Alexander Torres-Gomez} 
\address{Instituto de Matem\'{a}ticas \\ Universidad de Antioquia \\ Calle 62 $\#$ 52-59 \\ Medell\'{i}n \\ Colombia}
\email{galexander.torres@udea.edu.co}

\begin{document}

\begin{abstract}
We develop Geometric Invariant Theory for affine superschemes under the action of reductive algebraic supergroups, formulating the theory in terms of coordinate Hopf superalgebras in order to accommodate anticommuting and nilpotent variables. Within this algebraic setting, we construct affine superquotients and establish their basic structural properties under suitable hypotheses. We then prove a supergeometric analogue of the Hilbert--Mumford criterion, giving a numerical test for the semistability and stability of points under natural constraints on the supergroup coaction. We show that this numerical criterion can be used to construct the GIT superquotient of these superschemes, and we give examples illustrating how the supergeometric GIT superquotient differs from its classical counterpart.\\

\noindent \textsc{Key words:} Affine superschemes, Hopf superalgebras, super Harisch-Chandra pairs, GIT superquotient, super invariants, supercharacters.
\end{abstract}

\subjclass[2020]{Primary: 14L24, 14A24, 16S38; Secondary: 81T75}

\maketitle

\section{Introduction}
Geometric Invariant Theory, originally developed by Mumford \cite{Mumford1961,Mumford1965} and henceforth referred to as GIT, provides one of the most comprehensive frameworks for systematically analyzing geometric invariants. At its core, GIT addresses the construction of quotient spaces for algebraic reductive group actions on varieties or schemes using the machinery of algebraic geometry. Today, it stands as a cornerstone in the theory of moduli spaces and remains central to research in commutative algebra and algebraic geometry.

Over the last three decades, the scope of GIT has expanded significantly, driven largely by its interactions with representation theory and noncommutative algebraic geometry. A major catalyst for this progress was the work of King \cite{King1994}, who used GIT to construct and classify moduli spaces of quiver representations. This foundational approach subsequently revealed that many classical geometric objects, such as projective toric varieties \cite{CrawSmith2008}, quiver flag varieties \cite{Craw2011}, and Mori Dream Spaces \cite{CrawWinn2013}, could be explicitly realized as fine moduli spaces of quivers. Parallel to these developments, noncommutative algebraic geometry was shaped by the introduction of noncommutative crepant resolutions by Van den Bergh \cite{VanDenBergh2004, VanDenBergh2004NCCR}. These fields connect through the insight that determining the moduli spaces of representations for such noncommutative crepant resolutions allows one to extract genuine resolutions of singularities. This theme also encompasses the study of the McKay correspondence \cite{CassensSlodowy1998, CrawIshii2004} and dimer models \cite{IshiiUeda2015}, as well as subsequent combinatorial refinements in noncommutative toric geometry involving cellular resolutions from superpotentials \cite{CrawQuinteroVelez2012} and geometric Reid's recipe \cite{BocklandtCrawQuinteroVelez2015}. More broadly, this line of research has continued to develop through numerous related contributions \cite{hoskins2018parallels,hoskins2018stratifications, craw2018multigraded,bellamy2020birational,craw2021punctual,bellamy2026birational}. At present, these intertwined strategies stand at the core of contemporary milestones in homological algebra, most notably within the homological minimal model program \cite{Wemyss2018} and the resolution of the noncommutative Bondal--Orlov conjecture \cite{IyamaWemyss2013}. Ultimately, these research lines have advanced our understanding of derived categories of coherent sheaves, where the study of autoequivalences and derived symmetries via the variation of GIT quotients \cite{HalpernLeistner2015, HalpernLeistnerShipman2016, BallardFaveroKatzarkov2019} has given a clearer geometric and homological picture of algebraic quotient constructions.

This same noncommutative perspective has, in recent years, extended into a rather different direction: that of superspaces and superschemes, which have gained increasing prominence within algebraic geometry. These objects arise as natural extensions of the concept of a supermanifold, introduced by physicists in the 1970s to provide a rigorous mathematical framework for supersymmetry. Intuitively, a supermanifold generalizes the classical notion of manifold by incorporating, alongside the usual ``bosonic'' coordinates, ``fermionic'' coordinates that anticommute. This idea was subsequently formalized and extended to the algebraic setting through the foundational work of Leites, Manin, Kapranov, Polishchuk, Bruzzo, Hernández-Ruipérez, and others \cite{Leites1974,Manin1988,Manin1991,KapranovVasserot2011,BruzzoHernandezRuiperezPolishchuk2023}, giving rise to the theory of superschemes, which today occupies a central place within noncommutative geometry and continues to deepen our understanding of space and geometry.

The main objective of this article is to lay down the foundational constructions for a systematic implementation of GIT for affine superschemes. The presence of anticommuting and nilpotent variables means that points of a superscheme cannot, in general, be treated as ordinary points of a topological space, so actions and quotients must instead be formulated intrinsically at the level of coordinate superalgebras. To ensure that affine quotients and algebraic supergroup actions remain well-defined, it becomes essential to resort to a purely algebraic framework. Emulating the approach initiated by King, we formulate our theory entirely in terms of coactions of the Hopf superalgebras associated to such algebraic supergroups.

Within this framework, we formally define affine superquotients and establish their corresponding theory, proving their existence and core algebraic properties under appropriate conditions. Building upon this foundation, the primary result of this work is the formulation and proof of a supergeometric analogue of the Hilbert--Mumford criterion, which translates the stability and semistability of topological points into a purely numerical test. Proving this main result, however, is far from a direct generalization of the classical setting, as the coupling between even and odd variables introduces serious obstructions to evaluating limits. To overcome these difficulties, we introduce two structural conditions, ``odd-infinitesimally decoupled coactions'' and ``polynomial coextensions'', which delimit the setting in which the relevant geometric arguments can be carried out. By imposing these conditions, we show how Mumford's and King's original insights can be adapted into a rigorous, self-contained supergeometric machinery.

This work is motivated by the need to extend King's approach to superquivers. This direction follows naturally from the homological behavior of projective superspaces. In a forthcoming companion paper \cite{QVRHTG2026art1}, we establish a Beilinson-type theorem showing that an exceptional collection of line bundles generates the derived category of coherent sheaves on a projective superspace, and we show that the resulting ``quiver of sections'' is, in its algebraic structure, precisely a superquiver. Building on this result, constructing the corresponding moduli spaces of superquiver representations is an ongoing project of the authors \cite{QVRHTG2026art2}. As part of the same broader program, we are also developing explicit constructions of several key supergeometric objects, including Hilbert superschemes, supergeometric analogues of Kronheimer's construction, and Nakajima quiver supervarieties. The algebraic and numerical results established here lay the necessary groundwork for all of these projects and are the basis on which they depend.

A second motivation comes from string theory, specifically from gauged linear sigma models (GLSMs). In the physics literature \cite{GuZou2019,Zou2025,ErLiuTan2026}, there is a well-known duality between the GLSM for a classical quintic Calabi–Yau hypersurface and the GLSM whose target space is the projective superspace $\PP^{4\vert{}1}$, suggesting that many Calabi–Yau varieties have supergeometric duals. Mathematically, GLSMs have been formalized by Fan, Jarvis, and Ruan \cite{FanJarvisRuan2018} and by Favero and Kim \cite{FaveroKim2020} using algebraic quotients via GIT. To extend these mathematical formulations to the super setting and establish such dualities with complete geometric rigor, a well-defined theory of GIT superquotients becomes indispensable. The framework introduced in this article provides the necessary foundations required to systematically implement these supergeometric quotients.

In relation to the current literature, we note that while this article was in preparation, Amrutiya and Dubey \cite{AmrutiyaDubey2026} independently introduced a partial development of GIT for superschemes. Their approach focuses on the construction of quotients under the action of ordinary algebraic groups, alongside stability criteria and examples tailored to representations of ordinary quivers in super vector spaces. This differs substantially from our formulation, where we develop the theory of quotients under the action of general reductive algebraic supergroups. In fact, the framework presented in \cite{AmrutiyaDubey2026} is recovered as the special case where the odd dimension of the acting supergroup vanishes. Our construction is therefore a strict generalization of theirs.

The remainder of this article is organized as follows. Section~\ref{sec:2} reviews the necessary preliminaries on affine superschemes, Hopf superalgebras, and algebraic supergroups, and fixes the notation and conventions used throughout the paper; for each topic, references are provided for further detail. Section~\ref{sec:3} develops the core algebraic framework of our setup by formalizing algebraic supergroup actions as coactions on coordinate superalgebras. From there, we introduce invariant subsuperalgebras and use them to construct naive affine superquotients, illustrating along the way several phenomena that have no counterpart in the ordinary setting. Finally, Section~\ref{sec:4} contains the main contributions of this work: we extend King's classical GIT construction to the supergeometric setting. We generalize the notion of relative invariants associated to a character, adapt the conditions for semistability and stability, and prove a numerical criterion that allows us to explicitly construct the GIT superquotient of affine superschemes, with further examples contrasting these new phenomena with the ordinary case.


\section{Background material}\label{sec:2}
This section collects the preliminary material needed throughout the paper. We briefly review the basic notions of affine superschemes, Hopf superalgebras, and affine algebraic supergroups, fixing notation and conventions that will be used in the sequel.

\subsection{Superalgebraic setup}
We start with the basic superalgebraic structures: superalgebras together with the necessary super linear algebra. The material presented here is standard, and we keep the discussion brief. 

Let $k$ be a fixed algebraically closed field. All objects considered here and in what follows are defined over $k$. By a \emph{super vector space}, we mean a $\mathbb{Z}_2$-graded vector space $V = V_0 \oplus V_1$. The elements of $V_0$ are called \emph{even} and those of $V_1$ \emph{odd}. Such elements are said to be \emph{homogeneous}, and for a homogeneous element $v$ we write $\lvert v \rvert$ for its parity. If $V_0$ and $V_1$ have dimensions $d_0$ and $d_1$, respectively, we say that $V$ has dimension $d_0 \vert d_1$. For super vector spaces $V$ and $V'$, the morphisms from $V$ to $V'$ are linear maps from $V$ to $V'$ which preserve the gradings. Such maps are often called \emph{even linear maps}. They form a vector space, denoted by $\Hom(V,V')$. We write $\sVect_k$  for the category of super vector spaces and even linear maps.

One of the fundamental features of the category $\sVect_k$ is that it carries a natural monoidal structure. For super vector spaces $V$ and $V'$, their tensor product is $V \otimes_k V'$ with homogeneous components
\begin{equation}
(V \otimes_k V')_0 = (V_0 \otimes_k V'_0) \oplus (V_1 \otimes_k V'_1), \quad (V \otimes_k V')_1 = (V_0 \otimes_k V'_1) \oplus (V_1 \otimes_k V'_0).
\end{equation}
Associated with this tensor product, via the usual adjunction, is the so-called internal $\Hom$, denoted by $\underline{\Hom}(V,V')$. This is the super vector space of all linear maps from $V$ to $V'$, where the even maps preserve the grading and the odd maps reverse it.

By a \emph{superalgebra}, we mean an algebra object $A$ in the category $\sVect_k$. Explicitly, this means that $A = A_0 \oplus A_1$ is a super vector space equipped with an associative product and a unit such that $A_i A_j \subseteq A_{i+j}$ for all $i,j \in \ZZ_2$. We say that $A$ is \emph{commutative} if $ab=(-1)^{\lvert a\rvert\lvert b\rvert}ba$ for all homogeneous elements $a,b \in A$. A commutative superalgebra $A$ is said to be \emph{local} if it has a unique maximal graded ideal $\mfrak$; equivalently, $A_0$ is a local algebra with maximal ideal $\mfrak \cap A_0$. It is said to be of \emph{finite type} if its even part $A_0$ is a finitely generated algebra and its odd part $A_1$ is a finitely generated $A_0$-module. For superalgebras $A$ and $A'$, the morphisms from $A$ to $A'$ are even linear maps which preserve the product and the unit. Commutative superalgebras are the only ones that will concern us here, and we write $\sAlg_k$ for their category.

Let $A$ be a commutative superalgebra. An \emph{$A$-supermodule} is a super vector space $M = M_0 \oplus M_1$ which is an $A$-module in the usual sense and, in addition, the action satisfies $A_i \cdot M_j \subseteq M_{i+j}$ for all $i,j \in \ZZ_2$. Morphisms of $A$-supermodules are even linear maps that commute with the $A$-action. An $A$-supermodule $M$ is said to be \emph{free} if it is free as an $A$-module with a basis consisting of homogeneous elements. The most important examples are the free $A$-supermodules $A^{m \vert n}$ of finite rank $m \vert n$, where $m$ and $n$ are non-negative integers. Such a supermodule has a homogeneous basis of $m$ even and $n$ odd elements, with even and odd components $(A^{m \vert n})_0 = A_0^{m} \oplus A_1^{n}$ and $(A^{m \vert n})_1 = A_1^{m} \oplus A_0^{n}$. Morphisms between different free supermodules $A^{m \vert n}$ can, as usual, be described by matrices.

We shall recall one more notion. By a \emph{Lie superalgebra}, we mean a Lie algebra object $\gfrak$ in the category $\sVect_k$. This amounts to a super vector space $\gfrak = \gfrak_0 \oplus \gfrak_1$ together with an even bilinear map $[,]\colon \gfrak \otimes_k \gfrak \to \gfrak$, termed the Lie bracket, which satisfies
\begin{equation}
[x,y] = - (-1)^{\lvert x \rvert \lvert y \rvert} [y,x],
\end{equation}
and
\begin{equation}
(-1)^{\lvert x \rvert \lvert z \rvert} [x,[y,z]] + (-1)^{\lvert y \rvert \lvert z \rvert} [z,[x,y]] + (-1)^{\lvert x \rvert \lvert y \rvert} [y,[z,x]] = 0,
\end{equation}
for all homogeneous elements $x,y,z\in\gfrak$. These identities are the natural modifications of the antisymmetry and Jacobi identities for an ordinary Lie algebra, designed to accommodate the $\ZZ_2$-grading. It should be noted, though, that this definition is most commonly used when Lie superalgebras are considered independently of Lie supergroups. Later, we will recall affine algebraic supergroups, which provide geometric realizations of Lie superalgebras, and see how these two notions are naturally related. This is all we will need here.

\subsection{Affine superschemes}\label{sec:2.2}
We now proceed to recall affine superschemes. For a comprehensive treatment, we refer the reader to \cite{Manin1988, Manin1991,CarmeliCastonFioresi2011, BruzzoHernandezRuiperezPolishchuk2023}. Some preliminary notions are needed first.

By a \emph{superspace}, we mean a pair $X=(\lvert X \rvert,\Ocal_X)$, where $\lvert X \rvert$ is a topological space and $\Ocal_X$ is a sheaf of superalgebras such that the stalk $\Ocal_{X,x}$ is a local superalgebra for all $x \in \lvert X \rvert$. A \emph{morphism} $f \colon X \to Y$ of superspaces is a pair $f=(\lvert f \rvert,f^{\#})$, where $\lvert f \rvert \colon \lvert X \rvert \to \lvert Y \rvert$ is a continuous map and $f^{\#} \colon \Ocal_Y\to \lvert f \rvert_*\Ocal_X$ is a morphism of sheaves satisfying $f_x^{\#}(\mfrak_{\lvert f \rvert (x)})\subseteq\mfrak_x$, where $\mfrak_{\lvert f \rvert(x)}$ and $\mfrak_x$ denote the unique maximal ideals of the stalks $\Ocal_{Y,\lvert f \rvert(x)}$ and $\Ocal_{X,x}$, respectively. Among superspaces, the objects of primary interest are \emph{superschemes}. These are defined to be the superspaces $X=(\lvert X\rvert,\Ocal_X)$ such that $(\lvert X \rvert,\Ocal_{X,0})$ is a scheme and $\Ocal_{X,1}$ is a quasicoherent sheaf of $\Ocal_{X,0}$-modules. Morphisms of superschemes are simply morphisms of the underlying superspaces. We write $\sSch_k$ for the category of superschemes and their morphisms.

We next recall the notion of the \emph{spectrum} of a commutative superalgebra $A$. This is the superscheme $\sSpec A = (\lvert \sSpec A \rvert, \Ocal_{\sSpec A})$, defined as follows. Its underlying topological space is $\lvert \sSpec A\rvert=\Spec A_0$, endowed with the Zariski topology. The structure sheaf $\Ocal_{\sSpec A}$ is the sheaf of commutative superalgebras whose sections over each basic open subset $D(f)\subseteq \lvert\sSpec A\rvert$, with $f\in A_0$, are given by
\begin{equation}
\Ocal_{\sSpec A} (D(f))=A_f=A\otimes_{A_0}(A_0)_f.
\end{equation}
The latter equality follows from the natural $A_0$-module structure on $A$ induced by the inclusion $A_0\hookrightarrow A$, under which $A_f$ inherits its natural $\mathbb Z_2$-graded $A_0$-module structure. For every prime ideal $\pfrak\in \lvert\sSpec A\rvert$, the stalk of $\Ocal_{\sSpec A}$ is the localization
\begin{equation}
\Ocal_{\sSpec A,\pfrak}= A_{\pfrak} =\left\{ \frac{a}{s} \: \Big\vert \: a\in A, s\in A_0\setminus\pfrak\right\},
\end{equation}
which is a local superalgebra with unique maximal ideal $\mfrak_{\pfrak} \oplus (A_{\pfrak})_1$. In fact, $A_{\pfrak}$ is canonically isomorphic to the direct limit $\varinjlim_{f\notin\pfrak}A_f$, where the directed system ranges over all $f\in A_0$ with $f\notin\pfrak$. This shows that the above identification agrees with the usual construction of stalks. 

We should point out that while the construction of the spectrum is purely algebraic and holds for any abstract commutative superalgebra, assuming $A$ to be of finite type ensures that the resulting superscheme possesses appropriate geometric properties. In this setting, the underlying topological space $\lvert \sSpec A \rvert = \Spec A_0$ becomes a classical affine scheme of finite type, thereby establishing a well-behaved Noetherian framework for the theory.

A superscheme $X$ is said to be \emph{affine} if there exists a commutative superalgebra $A$ such that $X\cong\sSpec A$. In this case, we define $k[X] = A$ and call it the \emph{coordinate superalgebra} of $X$. Alternatively, and more abstractly, an affine superscheme $X$ may be viewed as a functor
\begin{equation}
   X \colon \sAlg_k \longrightarrow \Sets,
\end{equation}
where $\Sets$ denotes the category of sets, that is represented by its coordinate superalgebra $k[X]$. That is to say, $X$ assigns to each commutative superalgebra $B$ the set
\begin{equation}
X(B) = \Hom_{\sAlg_k}(k[X], B).
\end{equation}
 This is the functor of points perspective, included here for later reference.
 
When the coordinate superalgebra $k[X]$ of an affine superscheme $X$ is of finite type, the superscheme itself is said to be \emph{algebraic}. This terminological distinction materializes the geometric framework discussed above, emphasizing that the condition of finite type shifts our focus from abstract categorical functors to concrete supergeometric spaces with the finiteness properties characteristic of algebraic geometry.

We supplement this framework with the following natural extensions of standard scheme-theoretic definitions, which will be deployed later in our study. An open subset $U \subseteq X$ of an affine superscheme is called \emph{affine} if there exists an even element $f$ of its coordinate superalgebra  $k[X]$ such that $U$ coincides with the basic open set $D(f)$. Accordingly, a morphism of affine superschemes $f \colon X \to Y$ is said to be \emph{affine} if the preimage of every affine open subset is again an affine open subset. Clearly, since the underlying topological space of an affine superscheme coincides exactly with the classical Zariski spectrum of its even component, an open subset is affine in this supergeometric sense if and only if its underlying topological counterpart is affine in the classical sense. In consequence, a morphism of affine superschemes is affine if and only if the induced continuous map between their underlying topological spaces maps preimages of classical affine open sets to classical affine open sets.

As our last foundational remark, we emphasize that the construction of the spectrum of a commutative superalgebra is functorial. Indeed, given a morphism of commutative superalgebras $\phi \colon A\to A'$, there is an induced morphism of superschemes $\phi^{*} = (\lvert \phi^{*} \rvert, \phi^{* \#}) \colon \sSpec A'\to\sSpec A$ , defined as follows. Let $\phi_0\colon A_0\to A'_0$ denote the restriction of $\phi$ to the even part. As in the classical setting, $\phi_0$ induces a continuous map $\lvert \phi^{*} \rvert\colon\lvert\sSpec A'\rvert\to\lvert\sSpec A\rvert$ given by
\begin{equation}
\lvert \phi^{*} \rvert(\qfrak)=\phi_0^{-1}(\qfrak)
\end{equation}
The morphism of structure sheaves $\phi^{* \#}\colon\Ocal_{\sSpec A}\to\lvert\phi^*\rvert_*\Ocal_{\sSpec A'}$ is induced by the localization morphisms, where for every $f\in A_0$ the morphism $\phi$ naturally induces a morphism $A_f\to A'_{\phi_0(f)}$ given by
\begin{equation}
\frac{a}{f^n}\longmapsto\frac{\phi(a)}{\phi_0(f)^n}. 
\end{equation}
At the level of stalks, for every $\qfrak\in\lvert\sSpec A' \rvert$ and $\pfrak=\phi_0^{-1}(\qfrak)$, the induced morphism $A_{\pfrak}\to A'_{\qfrak}$ is given by
\begin{equation}
\frac{a}{s}\longmapsto\frac{\phi(a)}{\phi_0(s)}.
\end{equation}
The preimage of the maximal ideal $\mfrak_{\qfrak}\oplus(A'_{\qfrak})_1$ is precisely $\mfrak_{\pfrak}\oplus(A_{\pfrak})_1$, so that $\phi^{*\#}_{\pfrak}$ is a local morphism. Consequently, the assignment sending $A$ to $\sSpec A$ extends to a contravariant functor
\begin{equation}
\sSpec\colon\sAlg_k\longrightarrow\sSch_k,
\end{equation}
whose essential image is precisely the full subcategory of affine superschemes.

\subsection{Hopf superalgebras}
We now turn to Hopf superalgebras. A detailed treatment can be found in \cite{GouldZhangBracken1993,DongHuang2011,Westra2009}. We refer the reader to these works for any terms not explicitly defined here.

First, one exploits the monoidal structure of $\sVect_k$ to describe superalgebras in terms of arrows and diagrams. Thus, a superalgebra is a triple $(A,m,u)$, where $A$ is a super vector space, and $m \colon A \otimes_k A \to A$ and $u \colon k \to A$ are even linear maps encoding the usual associativity and unit axioms, respectively. The commutativity of a superalgebra can be expressed via the graded flip $\sigma \colon A \otimes_k A \to A \otimes_k A$, given on homogeneous elements by $\sigma (a \otimes b) = (-1)^{\lvert a \rvert \lvert b \rvert} b \otimes a$, by requiring $m \circ \sigma = m$. A morphism of superalgebras is an even linear map preserving both $m$ and $u$ under the induced tensor structures. Reversing the arrows in these definitions dually yields the notion of a supercoalgebra $(C,\Delta,\varepsilon)$, where the even linear maps $\Delta \colon C \to C \otimes_k C$ and $\varepsilon \colon C \to k$ specify a coassociative coproduct and a counit, respectively. Cocommutativity is similarly defined via $\sigma \circ \Delta = \Delta$ and morphisms of supercoalgebras are even linear maps preserving both $\Delta$ and $\varepsilon$.

A \emph{superbialgebra} is a quintet $(B,m,u,\Delta,\varepsilon)$, where $(B,m,u)$ is a superalgebra and $(B,\Delta,\varepsilon)$ is a supercoalgebra, such that the maps $\Delta$ and $\varepsilon$ are also morphisms of superalgebras. Such a superbialgebra is called a \emph{Hopf superalgebra} if it is additionally equipped with an even linear map $S \colon H \to H$ such that the following diagram commutes:
\begin{equation}
\begin{tikzcd}[row sep=2.5em, column sep=3.0em]
H \otimes_k H  \arrow[d,-{To[length=2.5pt, width=4pt]},"\id_H \otimes S"'] &
H \arrow[l,-{To[length=2.5pt, width=4pt]},"\Delta"'] \arrow[r,-{To[length=2.5pt, width=4pt]},"\Delta"] \arrow[d,-{To[length=2.5pt, width=4pt]},"u \circ \varepsilon"] &
H \otimes_k H \arrow[d,-{To[length=2.5pt, width=4pt]},"S \otimes \id_H"] \\
H \otimes_k H \arrow[r,-{To[length=2.5pt, width=4pt]},"m"] &
H &
H \otimes_k H  \arrow[l,-{To[length=2.5pt, width=4pt]},"m"'] 
\end{tikzcd}
\end{equation}
As is standard, the map $S$ is called the \emph{antipode} of $H$. A \emph{morphism of Hopf superalgebras} is an even linear map $\phi \colon H \to H'$ that is both a morphism of superalgebras and a morphism of supercoalgebras, and satisfies the compatibility condition $\phi \circ S = S' \circ \phi$. 

The antipode $S$ in a Hopf superalgebra $H$ fulfills this role of inversion. To see this, one may observe how it acts on group-like elements. Modeled after the classical setting, a non-zero even element $g \in H_0$ is called \emph{group-like} if $\Delta(g) = g \otimes g$ and $\varepsilon(g) = 1$. The set of group-like elements forms a unital semigroup under the product of $H$ because $\Delta $ is an algebra morphism preserving the unit. For any group-like element $g$, evaluating the defining diagram of the antipode immediately yields $g S(g) = 1_H = S(g) g$, which forces $g$ to be algebraically invertible with $S(g) = g^{-1}$. Since the inverse of a group-like element must remain group-like, we conclude that the group-like elements constitute a genuine group inside the group of units $H_0^{\times}$, within which $S$ provides the inversion.

A \emph{super bi-ideal} in a Hopf superalgebra $H$ is a super vector subspace $I$ that is a two-sided ideal of $H$ as a superalgebra and a two-sided coideal of $H$ as a supercoalgebra. A \emph{Hopf superideal} of $H$ is then a super bi-ideal that is stable under the action of the antipode, meaning that $S(I) \subseteq I$. If $I \subseteq H$ is a Hopf superideal, the quotient space $H/I$ naturally inherits the structure of a Hopf superalgebra with operations induced from $H$, explicitly given by $\bar{\Delta} (x + I) = \Delta(x) +  (I \otimes_k H + H \otimes_k I)$, $\bar{\varepsilon}(x + I) = \varepsilon(x)$ and $\bar{S}(x+I) = S(x) + I$. One can show that $\bar{S}$ is the only choice to turn $H/I$ into a Hopf superalgebra. Furthermore, if $\phi \colon H \to H'$ is a morphism of Hopf superalgebras, then the kernel of $\phi$ is a Hopf superideal of $H$. 

\subsection{Affine algebraic supergroups}\label{sec:2.4}
We now bring together the notions introduced above to define affine algebraic supergroups. Here we follow closely the treatment given in \S8.5 of \cite{Westra2009} and \S3 of \cite{MasuokaTakahashi2021}. Further details can be found in \cite{FioresiLledo2004,Fioresi2008,Zubkov2009,MasuokaZubkov2011,BovdiZubkov2023}. 

By an \emph{affine algebraic supergroup}, we mean a group object $G$ in the category of affine superschemes. Equivalently, it is a representable functor 
\begin{equation}
G \colon \sAlg_{k} \longrightarrow \Groups,
\end{equation}
where $\Groups$ denotes the category of groups. Thus, there exists a commutative superalgebra $k[G]$ such that, for every commutative superalgebra $A$,
\begin{equation}
G(A) = \Hom_{\sAlg_k}(k[G],A).
\end{equation}
Here, $k[G]$ is precisely the coordinate superalgebra of the affine superscheme representing $G$, and is required to be of finite type. In that case, $k[G]$ has a canonical structure of Hopf superalgebra where the coproduct $\Delta_G \colon k[G] \to k[G] \otimes_k k[G]$, counit $\varepsilon_G \colon k[G] \to k$ and antipode $S_G \colon k[G] \to k[G]$ are the unique morphisms dual to the structural morphisms defining the group structure on $G$. Conversely, every finitely generated commutative Hopf superalgebra arises in this way from a unique affine algebraic supergroup.

Let $G$ be an affine algebraic supergroup with representing Hopf superalgebra $k[G]$. Given a Hopf superideal $I \subseteq k[G]$, then, as we saw in the preceding subsection, the quotient $k[G]/I$ is again a Hopf superalgebra, and therefore represents an affine algebraic supergroup $H$. The quotient morphism $k[G]\to k[G]/I$ induces a morphism $H\to G$, identifying $H$ with a closed  supersubgroup of $G$. Conversely, every closed supersubgroup of $G$ arises in this way. Of particular importance is the underlying purely even algebraic  supergroup $G_{\uev}$, which is the closed supersubgroup corresponding to the  Hopf superideal $(k[G]_1)$ generated by the odd component $k[G]_1$ of $k[G]$. As a functor, $G_{\uev}(A)=G(A_0)$ for every commutative superalgebra $A$. The quotient $k[G]/(k[G]_1)$ is moreover a purely even Hopf superalgebra, thereby representing an ordinary affine algebraic group denoted by $G_0$, whose representable functor is given by the restriction of $G$ to the category of  commutative algebras.

We shall also need the following observation. If $G$ is an affine algebraic supergroup, the conjugation natural transformation $c \colon G \times G \to G$ induces, under duality, a morphism $c^{*} \colon k[G] \to k[G] \otimes_k k[G]$ of commutative superalgebras. A Hopf superideal $I \subseteq k[G]$ is said to be \emph{normal} if $c^{*}(I) \subseteq k[G] \otimes_k I$.  One can then show that the above correspondence restricts to a one-to-one correspondence between normal Hopf superideals of $k[G]$ and closed normal supersubgroups of $G$. 

Associated with any such closed normal supersubgroup $H$ of $G$, corresponding to the normal Hopf superideal $I_H \subseteq k[G]$, one defines the quotient functor $G/H$ by $(G/H)(A) = G(A)/H(A)$ for every commutative superalgebra $A$. It is a nontrivial result that $G/H$ is again an affine algebraic supergroup. Its representing Hopf superalgebra $k[G/H]$ is naturally identified with the subsuperalgebra of coinvariants of $k[G]$ with respect to the quotient morphism $r_H \colon k[G] \to k[G]/I_H = k[H]$. More explicitly, this is
\begin{equation}
k[G/H] = \{ a \in k[G] \mid ((\id_{k[G]} \otimes r_H) \circ \Delta_G)(a) = a \otimes 1 \}.
\end{equation}
This quotient construction, with all its ingredients, will be particularly useful later in the paper.

At this point, we highlight another way of describing affine algebraic supergroups that will be relevant for our discussion. By a \emph{super Harish--Chandra pair}, we mean a pair $(G_0,\gfrak)$ consisting of an affine algebraic group $G_0$ and a Lie superalgebra $\gfrak=\gfrak_0\oplus\gfrak_1$ with $\Lie G_0=\gfrak_0$, together with a representation of $G_0$ on $\gfrak$ that preserves the Lie bracket and whose differential gives the adjoint action of $\gfrak_0$ on $\gfrak$. This notion is closely tied to the structure of affine algebraic supergroups. Indeed, let $G$ be an affine algebraic supergroup and let $k[G]$ be its coordinate Hopf superalgebra. By a $\varepsilon_G$-\emph{superderivation} we mean a homogeneous linear map $D\colon k[G]\to k$ satisfying the graded Leibniz rule
\begin{equation}\label{eq:2.16}
D(ab) = D(a)\varepsilon_G(b)+(-1)^{\lvert D \rvert \lvert a \rvert}\varepsilon_G(a)D(b)
\end{equation}
for all homogenoeus elements $a, b \in k[G]$. The Lie superalgebra of $G$, denoted $\Lie G$, is by definition the superspace of all such $\varepsilon_G$-superderivations. For homogeneous $D,D'\in\Lie G$, the Lie bracket is given by 
\begin{equation}\label{eq:2.17}
[D,D'] = (D \otimes D') \circ \Delta_G - (-1)^{\lvert D \rvert \lvert D' \rvert} (D' \otimes D) \circ \Delta_G.
\end{equation}
This algebraic description of $\Lie G$ agrees with its interpretation as the space of ``infinitesimal elements'' of $G$ at the identity. Now, the group $G$ itself acts naturally on $\Lie G$. This action is induced by conjugation and gives a representation of $G_0$ on $\Lie G$ that preserves the Lie bracket, and its differential gives the adjoint action of $\Lie G_0$ on $\Lie G$. Thus, $G$ gives rise to the super Harish--Chandra pair $(G_0,\Lie G)$. Conversely, if one is given a super Harish--Chandra pair $(G_0, \gfrak)$, one can associate to it an affine algebraic supergroup $G$. The underlying super vector space of its coordinate superalgebra $k[G]$ is defined as $k[G]=k[G_{0}]\otimes _{k}\medwedge^{\sbullet} \gfrak_{1}^{*}$, where $\medwedge^{\sbullet} \gfrak_{1}^{*}$ is the exterior algebra over the dual of $\gfrak_{1}$. Its product and coproduct are uniquely prescribed by the Hopf algebra structure of $k[G_0]$, the wedge product of $\medwedge^{\sbullet} \gfrak_{1}^{*}$, and the adjoint action of $G_{0}$ on $\mathfrak{g}_{1}$.

Another notion we will need is that of a character. To define it, let $\GG_m$ denote the multiplicative supergroup represented by the Hopf superalgebra $k[\GG_m] = k[t,t^{-1}]$, where $t$ is even, with coproduct, counit and antipode given by $\Delta_{\GG_m}(t)=t\otimes t$, $\varepsilon_{\GG_m}(t)=1$ and $S_{\GG_m}(t)=t^{-1}$. Given an affine algebraic supergroup $G$, by a \emph{character} of $G$ we mean a morphism of Hopf superalgebras
\begin{equation}
\chi^{*}\colon k[\GG_m]\longrightarrow k[G].
\end{equation}
Notice that, since $k[\GG_m]$ is generated by the invertible even element $t$, specifying such a morphism is equivalent to specifying an invertible even element $\chi=\chi^{*}(t)\in k[G]_0$. Compatibility with the Hopf superalgebra structure then implies that $\Delta_G(\chi)=\chi\otimes\chi$ and $\varepsilon_G(\chi)=1$. This means that $\chi$ is a group-like element of $k[G]$. Conversely, every group-like element of $k[G]$ determines a unique character of $G$. Thus, characters of $G$ are in one-to-one correspondence with the group-like elements of $k[G]$. 

There is one last notion that will also play a role in what follows. By a \emph{one-parameter subgroup} of an affine algebraic supergroup $G$ we mean a morphism of Hopf superalgebras
\begin{equation}
\lambda^{*}\colon k[G]\longrightarrow k[\GG_m].
\end{equation}
One should note that, since $k[\GG_m]$ is purely even, every such morphism necessarily vanishes on the odd component $k[G]_1$ of $k[G]$. It therefore annihilates the Hopf superideal $(k[G]_1)$, and hence factors uniquely through the quotient Hopf superalgebra $k[G_0] = k[G]/(k[G]_1)$. Conversely, every one-parameter subgroup $\lambda_0\colon\GG_m\to G_0$ in the ordinary sense determines a one-parameter subgroup of $G$. Indeed, the corresponding Hopf algebra morphism $\lambda_0^{*}\colon k[G_0]\to k[\GG_m]$ induces, via the quotient morphism $k[G]\to k[G_0]$, a unique Hopf superalgebra morphism $\lambda^{*}\colon k[G]\to k[\GG_m]$. Thus, one-parameter subgroups of $G$ are in one-to-one correspondence with ordinary one-parameter subgroups of $G_0$.

\subsection{Two fundamental examples}\label{sec:2.5}
We conclude this section with a couple of primary examples illustrating the notions introduced above. These will serve as our basic test cases throughout the paper. Full details may be found in the references cited in the preceding subsections.

First, we consider the affine superspace $\AA^{m \vert n}$ of dimension $m \vert n$, which provides the simplest example of an affine superscheme. It is defined to be the functor
\begin{equation}
\AA^{m \vert n} \colon \sAlg_k \longrightarrow \Sets
\end{equation}
that assigns to each commutative superalgebra $A$ the even part of the free $A$-supermodule $A^{m \vert n}$. This functor is represented by the commutative superalgebra
\begin{equation}
k[\AA^{m \vert n}] = k[x_1, \dots, x_m, \theta_1, \dots, \theta_n],
\end{equation}
where the $x_i$ are even variables and the $\theta_j$ are odd variables. In fact, in view of the definition, $\AA^{m \vert n}(A) = A_0^m \oplus A_1^n$, so that an element of $\AA^{m \vert n}(A)$ consists of $m$ even and $n$ odd elements of $A$. Such a choice uniquely determines a morphism of commutative superalgebras  $k[\AA^{m \vert n}]  \to A$ by sending each $x_i$ to an even element and each $\theta_j$ to an odd element of $A$. We also remark here that, since $k[\AA^{m \vert n}]$ is manifestly of finite type, $\AA^{m \vert n}$ is indeed an algebraic affine superscheme.

Secondly, we consider an example of an affine algebraic supergroup, namely the odd multiplicative supergroup $\GG_m^{1\vert n}$. To define it, let $A$ be a commutative superalgebra and consider the group $A_0^\times$ of invertible elements of the even part $A_0$. We can then endow the product $A_0^\times \times A_1^n$ with a group structure by means of the rule
\begin{equation}
(a, b_1, \dots, b_n) \cdot (a', b'_1, \dots, b'_n) = (aa', ab'_1 + a'b_1, \dots, ab'_n + a'b_n),
\end{equation}
for every pair of elements $(a, b_1, \dots, b_n), (a', b'_1, \dots, b'_n) \in A_0^\times \times A_1^n$. The group so obtained is denoted $\GG_m^{1\vert n}(A)$. This thereby yields a functor
\begin{equation}
\GG_m^{1\vert n} \colon \sAlg_k \longrightarrow \Groups,
\end{equation}
which sends each commutative superalgebra $A$ to $\GG_m^{1\vert n}(A)$. Such a functor is represented by the commutative superalgebra
\begin{equation}
k[\GG_m^{1\vert n}] = k[t,t^{-1},\xi_1,\dots,\xi_n],
\end{equation}
where $t$ is an even invertible variable and the $\xi_i$ are odd variables. The Hopf superalgebra structure on $k[\GG_m^{1\vert n}]$ is induced by the multiplication and inversion in $\GG_m^{1 \vert n}$. On the generators, the coproduct and counit take the form
\begin{equation}
\begin{gathered}
\Delta_{\GG_m^{1\vert n}}(t) = t \otimes t, \quad \Delta_{\GG_m^{1\vert n}}(\xi_i) = t \otimes \xi_i + \xi_i \otimes t,\\
\varepsilon_{\GG_m^{1\vert n}}(t) = 1, \quad \varepsilon_{\GG_m^{1\vert n}}(\xi_i) =0.
\end{gathered}
\end{equation}
As for the antipode, it is given by
\begin{equation}
S_{\GG_m^{1\vert n}}(t)=t^{-1}, \quad S_{\GG_m^{1\vert n}}(\xi_i) = -t^{-2} \xi_i.
\end{equation}
It should be apparent from the definition that the underlying affine algebraic group of $\GG_m^{1\vert n}$ is naturally identified with $\GG_m$. It should also be noted that the Lie superalgebra of $\GG_m^{1\vert n}$ has an even part generated by an even element $H$ and an odd part generated by odd elements $Q_1,\dots,Q_n$, with Lie brackets given by
\begin{equation}
[H,Q_i] = Q_i, \quad [H,H] = 0, \quad  [Q_i,Q_j] = 0.
\end{equation}
In terms of the even and odd variables of $k[\GG_m^{1\vert n}]$, these generators are represented by the vector fields
\begin{equation}
H = t \frac{\partial}{\partial t}, \quad Q_i = t\frac{\partial}{\partial \xi_i}. 
\end{equation}
One may, too, consider characters of $\GG_m^{1\vert n}$ in the sense defined above. For instance, the simplest non-trivial character is $\chi_1^* \colon k[\GG_m] \to k[\GG_m^{1\vert n}]$, defined by $\chi_1^*(t)$. Upon composition with the quotient morphism $k[\GG_m^{1\vert n}] \to k[\GG_m]$, it yields the identity on $k[\GG_m]$, thereby inducing the identity character on the underlying algebraic group $\GG_m$. More generally, for each integer $m \in \ZZ$, one has the morphism $\chi_m^* \colon k[\GG_m] \to k[\GG_m^{1\vert{}n}]$, given by $\chi_m^*(t) = t^m$, which, when composed with the quotient morphism, likewise yields the character of weight $m$ on $\GG_m$. Consequently, none of these characters detect the odd variables of $k[\GG_m^{1\vert n}]$.


\section{Supergroup coactions and affine superquotients}\label{sec:3}
This section establishes the foundational algebraic framework required for our study of quotients in the super context. Given the affine nature of our objects, we formalize the geometric actions of affine algebraic supergroups primarily through their dual counterpart, namely supergroup coactions on coordinate superalgebras. This allows us to define the corresponding subsuperalgebras of invariants and construct the naive affine superquotients for affine superschemes.

\subsection{Supergroup coactions and invariants}\label{sec:3.1}
We begin by defining the notion of a supergroup coaction on a commutative superalgebra. Our immediate goal is to establish the precise conventions, terminology, and algebraic properties used in what follows.

Let $A$ be a commutative superalgebra and let $G$ be an affine algebraic supergroup represented by its Hopf superalgebra $k[G]$. By a \emph{coaction} of $k[G]$ on $A$, we mean a morphism of superalgebras $\rho \colon A \to A \otimes_k k[G]$ satisfying the coassociativity and counitality conditions, which state respectively that
\begin{equation}
\begin{gathered}
(\rho \otimes \id_{k[G]}) \circ \rho = (\id_A \otimes \Delta_G) \circ \rho, \\ 
(\id_A \otimes \varepsilon_G) \circ \rho = \id_A.
\end{gathered}
\end{equation}
The first condition ensures compatibility with the supergroup multiplication, whereas the second ensures that the identity element of the supergroup acts trivially. Within this framework, the appropriate notion of invariants for the affine supergroup action is defined entirely in terms of the underlying Hopf structure:
\begin{equation}\label{eq:3.2}
A^{G}=\{ a \in A \mid \rho(a) = a \otimes 1 \}.
\end{equation}
This subset $A^G$ constitutes a subsuperalgebra of $A$ and contains precisely the elements that remain unaltered under the coaction of $k[G]$. We shall refer to it as the \emph{superalgebra of invariants}.

It is worth noting that a coaction $\rho \colon A \to A \otimes_k k[G]$ naturally induces a conventional group action on the commutative superalgebra $A$. Indeed, the set of superalgebra morphisms $\xi \colon k[G] \to k$, formally known as the $k$-\emph{points} of the supergroup $G$, forms a group under the convolution product. Each such point yields a superalgebra endomorphism of $A$ given by
\begin{equation}
\sigma_{\xi} = (\id_A \otimes \xi) \circ \rho,
\end{equation}
and the assignment sending each $\xi$ to $\sigma_{\xi}$ defines an action of this group of points on $A$. This point-wise description will be employed later.

We now isolate a preliminary fact about the superalgebra of invariants, which underlies much of what follows.

\begin{lemma}\label{lem:3.1}
Let $G$ be an affine algebraic supergroup and $G_0$ its underlying affine algebraic group. For any coaction $\rho \colon A \to A \otimes_k k[G]$ on a commutative superalgebra $A$, one has $A^{G} \subseteq A^{G_0}$. 
\end{lemma}

\begin{proof}
Let $r_0 \colon k[G] \to k[G_0]$ be the canonical restriction morphism to the underlying algebraic affine group, corresponding to the projection onto the reduced part. The coaction of $G_0$ on $A$, denoted by $\rho_0$, is defined as the composition $\rho_0 = (\id_A \otimes r_0) \circ \rho$. If $a \in A^{G}$, then by definition $\rho(a) = a \otimes 1$. Applying the restriction map, we obtain
$$
\rho_0(a) = (\id_A \otimes r_0)(a \otimes 1) = a \otimes r_0(1) = a \otimes 1,
$$
which implies that $a \in A^{G_0}$, and consequently $A^{G} \subseteq A^{G_0}$, as desired.
\end{proof}

One might ask whether the reverse inclusion in Lemma~\ref{lem:3.1} also holds. The following example shows that this is not the case.

\begin{example}\label{ex:3.2}
Let us consider the odd multiplicative supergroup $\GG_{m}^{1 \vert 1}$. As detailed in \S\ref{sec:2.5}, it is represented by the Hopf superalgebra $k[\GG_m^{1\vert 1}] = k[t, t^{-1}, \xi]$, where $t$ is an even variable and $\xi$ is an odd variable, equipped with the structural coproducts $\Delta_{\GG_{m}^{1 \vert 1}}(t) = t \otimes t$ and $\Delta_{\GG_{m}^{1 \vert 1}}(\xi) = \xi \otimes t + t \otimes \xi$, counits $\varepsilon_{\GG_{m}^{1 \vert 1}}(t)=1$ and $\varepsilon_{\GG_{m}^{1 \vert 1}}(\xi) = 0$, and antipodes $S_{\GG_{m}^{1 \vert 1}}(t) = t^{-1}$ and $S_{\GG_{m}^{1 \vert 1}}(\xi) = - t^{-2} \xi$. We examine a coaction of $k[\mathbb G_m^{1\mid 1}]$ on the commutative superalgebra $k[x,y,\theta]$, which, as we know, is geometrically the coordinate superalgebra of the affine superspace $\mathbb A^{2\mid 1}$. On the generators, this coaction $\rho \colon k[x,y,\theta] \to k[x,y,\theta] \otimes_k k[t, t^{-1}, \xi]$ is prescribed by
$$
\begin{aligned}
\rho (x)&=x\otimes t,\\ 
\rho (y)&=y\otimes t^{-1},\\ 
\rho (\theta )&=\theta \otimes t+x\otimes \xi.
\end{aligned}
$$
Let us determine the full superalgebra of invariants under $\rho$ by solving the defining condition in \eqref{eq:3.2}. A generic homogeneous element of $k[x,y,\theta]$ can be written as $a = f(x,y) + g(x,y)\theta$, where $f$ and $g$ are polynomials in the even generators. Applying the full coaction $\rho$ to this element and expanding it yields
$$
\rho(a) = \rho(f(x,y)) + \rho(g(x,y) \theta). 
$$
Let us consider the first summand, $\rho(f(x,y))$. Each monomial $x^{i} y^{j}$ occurring in $f(x,y)$ contributes to the latter a term $x^{i} y^{j} \otimes t^{i-j}$. Hence, for such a term to occur in an invariant, we must have $i=j$. This means that the even part of the superalgebra of invariants is generated by the standard classical invariant $u = xy$. We now consider the second summand, $\rho(g(x,y) \theta)$. For a monomial $x^{i} y^{j} \theta$ occurring in $g(x,y) \theta$, the coaction gives
$$
\rho(x^{i} y^{j} \theta) = x^{i} y^{j} \theta \otimes t^{i-j+1} + x^{i+1} y^{j}  \otimes  t^{i-j} \xi.
$$
The term $x^{i} y^{j} \theta \otimes t^{i-j+1}$ can occur in an invariant only when $i- j + 1 =0$. The term $x^{i+1} y^{j}  \otimes  t^{i-j} \xi$, however, contains the odd variable $\xi$. Since no term involving $\xi$ occurs in $\rho(f(x,y))$, such terms cannot be cancelled by any contribution coming from $f(x,y)$. Moreover, distinct monomials $x^{i} y^{j}$ occurring in $g(x,y)$ give distinct monomials $x^{i+1} y^{j}$, so these terms cannot cancel among themselves either. Consequently, the invariance condition forces $g(x,y) = 0$. Putting all this together, the full superalgebra of invariants is
$$
k[x,y,\theta]^{\GG_m^{1 \vert 1}} = k[u].
$$
We now examine what happens when the coaction $\rho$ is projected onto the classical reduced subgroup  $(\GG_m^{1 \vert 1})_0 = \GG_m$ via the restriction map $r_0 \colon k[t,t^{-1},\xi] \to k[t,t^{-1}]$ that annihilates the odd variable $\xi$. As in the proof of Lemma~\ref{lem:3.1}, the resulting coaction is $\rho_0 = (\id_{k[x,y,\theta]} \otimes r_0) \circ \rho$ and, from the definition, is given on the generators by
$$
\begin{aligned}
\rho_0(x) &= x \otimes t, \\
\rho_0(y) &= y \otimes t^{-1}, \\
\rho_0(\theta) &= \theta \otimes t .
\end{aligned}
$$
With this reduced coaction in hand, the preceding calculation can be simplified considerably. The even part is generated by the same classical invariant $u = xy$, or exactly the same reason as before. The odd part, however, behaves differently. Under the reduced coaction, a monomial of the odd sector $x^{i} y^{j} \theta$ is mapped to
$$
\rho_0(x^{i}y^{j} \theta) = x^{i} y^{j} \theta \otimes t^{i - j +1}.
$$
Hence, just as before, invariance requires $i-j+1 = 0$. The difference is that the additional term involving $\xi$ is now absent, so this condition can actually be satisfied. The minimal solution in nonnegative integers is $i=0$ and $j=1$, so that the odd part is generated over $k[u]$ by the odd generator $\eta = y\theta$. We note that $\eta^2 = (y\theta)^2 = y^2\theta^2 = 0$, since $\theta$ is an odd variable. Furthermore, there are no additional relations between $u$ and $\eta$ beyond $\eta^2=0$. We conclude that the superalgebra of invariants is
$$
k[x,y,\theta]^{\GG_m} = k[u,\eta].
$$
From this, we see that $k[x,y,\theta]^{\GG_m}$ is strictly larger than $k[x,y,\theta]^{\GG_m^{1 \vert 1}}$, and hence is not contained in $k[x,y,\theta]^{\GG_m^{1 \vert 1}}$. Thus, the reverse inclusion in Lemma~\ref{lem:3.1} does not hold in general.
\end{example}

Since the inclusion in Lemma~\ref{lem:3.1} cannot in general be promoted to an equality, we instead ask whether equality can be expected for the even parts. This is indeed the case in Example~\ref{ex:3.2}, and we now investigate this phenomenon in general. So let $G$ be an affine algebraic supergroup with underlying affine algebraic group $G_0$. Write $\gfrak = \Lie G$ for the Lie superalgebra of $G$, and recall from \S\ref{sec:2.4} that $G$ can equivalently be described by the super Harish--Chandra pair $(G_0,\gfrak)$. If, as usual, $\gfrak_1$ denotes the odd part of $\mathfrak{g}$, an action of $\mathfrak{g}_{1}$ on a commutative superalgebra $A$ by odd superderivations is a linear map $\delta \colon \mathfrak{g}_1 \otimes_k A \to A$ satisfying
\begin{equation}\label{eq:3.4}
\delta(D \otimes ab) = \delta(D \otimes a) b + (-1)^{\lvert a \rvert} a\delta(D \otimes b)
\end{equation}
for every $D \in \mathfrak{g}_1$ and all homogeneous elements $a, b \in A$. The next result shows how the pair $(G_0,\gfrak)$, together with the notion just introduced, provides a convenient characterization of coactions of $k[G]$ on commutative superalgebras.

\begin{lemma}\label{lem:3.3}
Let $A$ be a commutative superalgebra. Then giving a coaction $\rho \colon A \to A \otimes_k k[G]$ is equivalent to giving a coaction $\rho_0 \colon A \to A \otimes_k k[G_0]$ together with an action by odd superderivations $\delta \colon \gfrak_1 \otimes_k A \to A$ satisfying the following conditions: \emph{(i)}  $\delta$ is compatible with $\rho _{0}$ via the relation
\begin{equation}
\rho_0 \circ \delta = (\delta \otimes \id_{k[G_0]}) \circ (\id_{\gfrak_1} \otimes m_{G_0}) \circ (\id_{\gfrak_1} \otimes \tau \otimes \id_{k[G_0]}) \circ (\ad^{*} \otimes \rho_0),
\end{equation}
where $\tau \colon k[G_0] \otimes_k A \to A \otimes_k k[G_0]$ is the canonical super-flip map, $m_{G_0} \colon k[G_0] \otimes_k k[G_0] \to k[G_0]$ is the product map, and $\mathrm{ad}^* \colon \gfrak_1 \to \gfrak_1 \otimes_k k[G_0]$ is the coaction induced by the adjoint representation of $G_{0}$ on $\gfrak_{1}$; \emph{(ii)} for every $D, D' \in \gfrak_1$ and $a \in A$,
\begin{equation}\label{eq:3.6}
\delta (D \otimes \delta(D' \otimes a)) + \delta (D' \otimes \delta(D \otimes a)) = (\id_{A} \otimes [D,D'])(\rho_0 (a)).
\end{equation}
\end{lemma}

\begin{proof}
Suppose we are given a coaction $\rho \colon A \to A \otimes_k k[G]$. Then, the coaction $\rho_0 \colon A \to A \otimes_k k[G_0]$ is simply the one defined by the restriction morphism $r_0 \colon k[G] \to k[G_0]$ in the proof of Lemma~\ref{lem:3.1}. To construct the action $\delta \colon \gfrak_1 \otimes_k A \to A$, recall that elements $D \in \gfrak_1$ act as $\varepsilon _{G}$-superderivations on $k[G]$. We define the linear map $\delta \colon \gfrak_1 \otimes_k A \to A$ by evaluating the second factor of the coaction, that is,
$$
\delta (D\otimes a)=(\id_{A}\otimes D)(\rho (a))
$$
for all $D \in \gfrak_1$ and $a \in A$. Since $\rho$ is a morphism of superalgebras and $D$ satisfies the graded Leibniz rule~\eqref{eq:2.16}, it follows immediately that $\delta$ satisfies relation~\eqref{eq:3.4}. The compatibility relation (i) is a direct consequence of the coassociativity of $\rho$ combined with the fact that the action of $G_{0}$ on $\gfrak_{1}$ dually manifests as the adjoint coaction $\ad^*$. Condition (ii) follows from the algebraic definition of the Lie bracket~\eqref{eq:2.17}; for odd elements, this bracket relation translates precisely into equation~\eqref{eq:3.6} on $A$.

Conversely, suppose we are provided with a $\rho_0 \colon A \to A \otimes_k k[G_0]$ and an action by odd superderivations $\delta \colon \gfrak_1 \otimes_k A \to A$ satisfying conditions (i) and (ii). Recall that, by the definition of the affine algebraic supergroup associated with the pair $(G_0, \gfrak)$, the underlying super vector space of $k[G]$ is given by $k[G_0] \otimes_k \medwedge^{\sbullet} \gfrak_1^*$. Using this decomposition, we define the linear map $\rho \colon A \to A \otimes_k k[G]$ by specifying how products of odd elements evaluate on the second tensor factor of the codomain. Explicitly, for every homogeneous element $a \in A$ and any elements $D_1, \dots, D_n \in \mathfrak{g}_1$, this evaluation is defined by the formula
$$
(\id_A \otimes D_1 \cdots D_n)\rho(a) = (-1)^{n\lvert a \rvert + \frac{n(n-1)}{2}} \rho_0 ( \delta(D_n \otimes \delta(D_{n-1} \otimes \cdots \otimes \delta(D_2 \otimes \delta(D_1 \otimes a)) \cdots )) ),
$$
where the sign is uniquely prescribed by the Koszul sign rule for shifting the odd operators $D_{i}$ past the homogeneous element $a$. Condition (i) ensures that this map $\rho$ is compatible with the coaction $\rho _{0}$, meaning that it is equivariant with respect to the action of the even subgroup $G_{0}$. Meanwhile, condition (ii) guarantees that the higher successive applications of $\delta$ are compatible with the relations of the exterior algebra $\medwedge^{\sbullet} \gfrak_1^*$, which ensures that $\rho$ is compatible with the Lie bracket structure given in~\eqref{eq:2.17}. Finally, because $\delta$ acts by superderivations, a straightforward inductive argument on the degree shows that $\rho$ preserves the product, making it a genuine morphism of superalgebras.  Coassociativity then follows from the combination of the coassociativity of the Hopf algebra $k[G_0]$ and the compatibility between $\delta$ and the adjoint coaction.
\end{proof}

Retaining the setup of the previous lemma, let $G$ be an affine algebraic supergroup with associated super Harish--Chandra pair $(G_0, \mathfrak{g})$. For a coaction $\rho \colon A \to A \otimes_k k[G]$ on a commutative superalgebra $A$, we shall ease the notation by writing $D \cdot a = \delta(D \otimes a)$ for all $D \in \mathfrak{g}_1$ and $a \in A$, where $\delta \colon \gfrak_1 \otimes_k A \to A$ is the corresponding action by odd superderivations. We then define the \emph{superalgebra of odd infinitesimal invariants} as
\begin{equation}\label{eq:3.7}
A^{\gfrak_1} = \{ a \in A \mid \text{$D \cdot a = 0$ for all $D \in \gfrak_1$}\}.
\end{equation}
It is a straightforward verification that $A^{\gfrak_{1}}$ is a subsuperalgebra of $A$. Crucially, this subsuperalgebra provides a refinement of the inclusion established in Lemma~\ref{lem:3.1} via the following result.

\begin{proposition}\label{prop:3.4}
Let $G$ be an affine algebraic supergroup and $(G_0, \mathfrak{g})$ its corresponding super Harish--Chandra pair. For any coaction $\rho \colon A \to A \otimes_k k[G]$ on a commutative superalgebra $A$, one has 
\begin{equation}
A^{G}  = A^{G_0} \cap A^{\gfrak_1}. 
\end{equation}
\end{proposition}

\begin{proof}
We first show that $A^{G}$ is contained in the intersection $A^{G_{0}} \cap A^{\gfrak_{1}}$. Let $a \in A^G$. By Lemma~\ref{lem:3.1}, it is already known that $a \in A^{G_0}$. On the other hand, since $a \in A^G$ means that $\rho(a) = a \otimes 1$, we can evaluate the action of any element $D \in \gfrak_1$ on $a$ using the construction given in the proof of Lemma~\ref{lem:3.3}, that is,
$$
D\cdot a=(\id_{A}\otimes D)(\rho(a)) =(\id_{A}\otimes D)(a\otimes 1)=a\otimes D (1).
$$
Since $D$ acts as an $\varepsilon _{G}$-superderivation, it annihilates the unit element, meaning $D(1) = 0$. Consequently, $D \cdot a = 0$ for all $D \in \gfrak_1$, which proves that $a \in A^{\gfrak_1}$. This completes the proof of the forward inclusion.

We next show that the intersection $A^{G_0} \cap A^{\gfrak_1}$ is contained in $A^{G}$. Let $a \in A^{G_0} \cap A^{\gfrak_1}$. Since $a$ belongs to $A^{G_{0}}$, we have $\rho_0(a) = a \otimes 1$. Furthermore, since $a$ is an element of $A^{\mathfrak{g}_{1}}$, we have by definition that $D \cdot a = 0$ for every element $D \in \gfrak_1$. To show that $a \in A^G$, we must verify that $\rho(a) = a \otimes 1$. Under the decomposition $k[G] = k[G_0] \otimes_k \medwedge^{\sbullet} \gfrak_1^*$, we evaluate the components of $\rho(a)$ by checking its contraction against any collection of odd elements $D_1, \dots, D_n \in \gfrak_1$. If $n = 0$, the evaluation reduces to the component on $k[G_0]$, which is given by $\rho_0(a) = a \otimes 1$. For any higher component where $n \geq 1$, the explicit formula from the proof of Lemma~\ref{lem:3.3} dictates that the evaluation is given by the nested expression
$$
(\id_A \otimes D_1 \cdots D_n)\rho(a) = (-1)^{n\lvert a \rvert + \frac{n(n-1)}{2}} \rho_0 ( D_0 \cdot ( D_{n-1} \cdot ( \cdots \cdot ( D_2 \cdot (D_1 \cdot a)) \cdots )))
$$
Since the innermost term is $D_1 \cdot a = 0$, and the action is linear, the entire nested expression collapses to zero. Because all higher projections onto the positive degrees of the exterior algebra vanish identically, the map $\rho(a)$ is completely concentrated in degree zero, yielding precisely $\rho(a) = \rho_0(a) = a \otimes 1$. This establishes that $a$ belongs to $A^{G}$, which concludes the proof.
\end{proof}

As a direct consequence of this proposition, we obtain the following description of the even component of the invariant superalgebra.

\begin{corollary}\label{cor:3.5}
With the notation and hypotheses as above, one has
\begin{equation}
(A^G)_0 = A_0^{G_0} \cap A_0^{\gfrak_1}.
\end{equation}
In particular, we obtain $(A^G)_0 = A_0^{G_0}$ if and only if $\gfrak_{1}$ acts trivially on $A_0^{G_0}$. 
\end{corollary}

\begin{proof}
By restricting the identity established in Proposition~\ref{prop:3.4} to the even degree components of the underlying super vector spaces, we obtain $(A^G)_0 = (A^{G_0})_0 \cap (A^{\mathfrak{g}_1})_0$. Since the coaction $\rho_0$ of the classical algebraic group $G_0$ preserves the $\ZZ_2$-grading, the even component of the superalgebra of invariants under $G_0$ coincides with $A_{0}^{G_{0}}$. Combining this with the fact that $A_{0}^{\gfrak_{1}}$ is by definition the even component of the superalgebra of odd infinitesimal invariants yields the desired intersection. The final criterion follows directly from the fact that the intersection equality $(A^G)_0 = A_0^{G_0}$ holds if and only if the inclusion $A_0^{G_0} \subseteq A_0^{\gfrak_1}$ is satisfied, which is precisely the requirement that every element in $A_{0}^{G_{0}}$ is annihilated by the operators of $\gfrak_{1}$, meaning that $\gfrak_{1}$ acts trivially on $A_{0}^{G_{0}}$.
\end{proof}

In view of the structural importance that the triviality condition in Corollary~\ref{cor:3.5} will play in what follows, we introduce a formal terminology. Given the superalgebra $A$ and the supergroup $G$ with its corresponding super Harish--Chandra pair $(G_0,\gfrak)$ as above, we shall say that a coaction $\rho \colon A \to A \otimes_k k[G]$ is \emph{odd-infinitesimally decoupled} if $\gfrak_{1}$ acts trivially on $A_{0}^{G_{0}}$. Geometrically, the idea is that the odd superderivations in $\gfrak_{1}$ fail to ``see'' the classical invariant subsuperalgebra $A_{0}^{G_{0}}$.

To establish our next result, which serves as a superalgebraic analogue of the classical Hilbert--Nagata theorem, it is necessary to introduce yet one more piece of terminology. We shall say that an affine algebraic supergroup $G$ is \emph{reductive} if its underlying affine algebraic group $G_0$ is reductive in the usual sense. In particular, this means that the notion of reductivity imposes no additional constraints on the odd directions.

\begin{proposition}\label{prop:3.6}
Let $A$ be a commutative superalgebra of finite type, and let $\rho \colon A \to A \otimes_k k[G]$ be an odd-infinitesimally decoupled coaction of a reductive affine algebraic supergroup $G$. Then the superalgebra of invariants $A^G$ is also of finite type.
\end{proposition}

\begin{proof}
We first examine the even part $(A^G)_0$. Since $A$ is of finite type, its even part $A_{0}$ is a finitely generated algebra, and is therefore a Noetherian ring. By the classical Hilbert--Nagata theorem, the algebra of invariants $A_{0}^{G_{0}}$ is a finitely generated algebra. Furthermore, because the coaction $\rho$ is odd-infinitesimally decoupled, Corollary~\ref{cor:3.5} guarantees that the even component of the superalgebra of invariants collapses precisely to this classical invariant algebra, meaning that $(A^G)_0 = A_0^{G_0}$. Thus, the even part $(A^G)_0$ is a finitely generated algebra.

We next consider the odd part $(A^{G})_1$. Since $A$ is of finite type, the odd part $A_1$ is a finitely generated $A_0$-module. The reductivity of the underlying algebraic group $G_0$ implies the existence of an exact Reynolds operator $R \colon A \to A^{G_0}$ which preserves the $\ZZ_{2}$-grading, restricting to a surjective morphism of $A_{0}^{G_{0}}$-modules from $A_1$ onto $(A^{G_0})_1$. By a classical theorem of Noether for modules under reductive group actions, it follows that $(A^{G_0})_1$ is a finitely generated $A_{0}^{G_{0}}$-module. Since $A_{0}^{G_{0}}$ is Noetherian, $(A^{G_0})_1$ is a Noetherian $A_{0}^{G_{0}}$-module. Finally, observing that by Lemma~\ref{lem:3.1} the odd component $(A^G)_1$ constitutes an $A_{0}^{G_{0}}$-submodule of the Noetherian module $(A^{G_0})_1$, we conclude that $(A^G)_1$ is also finitely generated over $A_0^{G_0} = (A^G)_0$. This shows that $A^{G}$ is of finite type, completing the proof.
\end{proof}

We close with a technical observation that will also play an important role in our study.

\begin{lemma}\label{lem:3.7}
Let $A$ be a commutative superalgebra equipped with a coaction $\rho \colon A \to A \otimes_k k[G]$ of an affine algebraic supergroup $G$, and let $S \subseteq (A^{G})_0$ be a multiplicatively closed subset. Then, there exists a natural isomorphism of commutative superalgebras
\begin{equation}
S^{-1} A^{G} \cong (S^{-1} A)^{G}.
\end{equation}
\end{lemma}

\begin{proof}
By definition, the superalgebra of invariants $A^{G}$ is the kernel of the linear map $\psi \colon A \to A \otimes_k k[G]$ defined by 
$$
\psi(a) = \rho(a) - a \otimes 1.
$$
Since the multiplicatively closed subset $S$ consists of even invariants, one has $\rho(s) = s \otimes 1$ for every $s \in S$, which ensures that the coaction uniquely extends to a well-defined coaction $\rho_S \colon S^{-1}A \to S^{-1}A \otimes_k k[G]$ on the localized superalgebra. The superalgebra of invariants under this extended action, $(S^{-1}A)^G$, is precisely the kernel of the linear map $\psi_S \colon S^{-1}A \to S^{-1}A \otimes_k k[G]$ given by
$$
 \psi_S\left( \frac{a}{s} \right) = \rho_S\left( \frac{a}{s} \right)  - \frac{a}{s} \otimes 1.
 $$
 To establish the required isomorphism, we apply the localization functor $S^{-1}$ directly to the map $\psi$. Since this functor is exact on the category of supermodules, it commutes with the taking of kernels, yielding a natural isomorphism of super vector spaces
$$
S^{-1} A^{G} = S^{-1} \ker \psi \cong \ker S^{-1} \psi. 
$$
Furthermore, because localization commutes with tensor products, we can canonically identify the codomain of the localized map $S^{-1}\psi$ so that it maps from $S^{-1}A$ into $S^{-1}A \otimes_k k[G]$ by sending any fraction $a/s$ to the element $(s \otimes 1)^{-1}\rho(a) - (a/s) \otimes 1$. By the construction of the extended coaction, this image matches the definition of $\psi_S(a/s)$ identically, which implies that 
$$
\ker S^{-1}\psi = \ker \psi_S = (S^{-1}A)^G.
$$
Bringing it all together, this chain of identifications yields the desired natural isomorphism $S^{-1} A^{G} \cong (S^{-1}A)^G$, which completes the proof.
\end{proof}

\subsection{Affine superquotients of affine superschemes}\label{sec:3.2}
We now turn to the definition and study of affine superquotients. Building on the developments of the preceding subsection, we proceed by mimicking the conventional prescription employed in ordinary algebraic geometry.

Let $X = \sSpec A$ be an affine superscheme, and let $G$ be an affine algebraic supergroup acting on $X$ via an odd-infinitesimally decoupled coaction $\rho \colon A \to A \otimes_k k[G]$. The affine superscheme $X {\sslash} G = \sSpec A^{G}$ is called the \emph{affine superquotient} of $X$ by $G$.

Explicitly, the structure of $X {\sslash} G$ follows directly from the general definition of the spectrum of a commutative superalgebra. First, since the coaction is odd-infinitesimally decoupled, Corollary~\ref{cor:3.5} ensures that $(A^{G})_0 = A_0^{G_0}$, so that the underlying topological space $\lvert \sSpec A^{G} \rvert$ is given by $\Spec A_0^{G_0}$. This shows that the topology of the superquotient is completely determined by the action of the reduced subgroup $G_{0}$ on the underlying classical scheme $\Spec A_0$. The structure sheaf of $X {\sslash} G$ is given, for each basic open subset $D(f) \subseteq \lvert \sSpec A^{G} \rvert$ with $f \in A_0^{G_0}$, by
\begin{equation}
\Gamma(D(f), \Ocal_{X {\sslash} G}) = (A^{G})_f.
\end{equation}
Since $f$ is an invariant element belonging to the even part, Lemma~\ref{lem:3.7} yields a natural isomorphism $(A^{G})_f \cong (A_{f})^{G}$. Consequently, the sections of the structure sheaf over basic open subsets can be interpreted as locally defined invariant functions on $X$. For every prime ideal $\pfrak \in \lvert \sSpec A^{G} \rvert$, the stalk of the sheaf is given by the localization
\begin{equation}
\Ocal_{X {\sslash} G, \pfrak} = (A^{G})_{\pfrak},
\end{equation}
which is a local commutative superalgebra. Once again, by Lemma~\ref{lem:3.7}, this stalk can be identified with $(A_{\pfrak})^{G}$, indicating that the local structure of the superquotient is entirely determined by invariant functions.

To make further progress in characterizing the geometric properties of the affine superquotient, we need to introduce some additional terminology. Let $X = \sSpec A$ be an affine superscheme equipped with an odd-infinitesimally decoupled coaction $\rho \colon A \to A \otimes_k k[G]$ of an affine algebraic supergroup $G$. A morphism $f \colon X \to Y$, where $Y = \sSpec B$ is another affine superscheme, is said to be $G$-\emph{invariant} if its dual morphism $f^* \colon B \to A$ satisfies $\rho(f^*(b)) = f^*(b) \otimes 1$ for all $b \in B$, which means that the image of $f^{*}$ is entirely contained within the superalgebra of invariants $A^G \subseteq A$. Similarly, a closed subset $Z \subseteq  X $ is defined to be $G$-\emph{invariant} if it is stable under the action induced by the underlying classical algebraic group $G_{0}$ on the topological base space $\lvert X \rvert$.

With all these elements in place, we now introduce the canonical map associated with the affine superquotient. So let $X = \sSpec A$ be an affine superscheme equipped with an odd-infinitesimally decoupled coaction $\rho \colon A \to A \otimes_k k[G]$ of an affine algebraic supergroup $G$. Consider the natural inclusion morphism of commutative superalgebras $A^G \hookrightarrow A$. By the functoriality of the spectrum construction, this algebraic inclusion induces a morphism of affine superschemes, which we denote by $\pi \colon X \to X {\sslash} G$, and which we shall call the \emph{affine superquotient morphism}. At the level of the underlying topological spaces, this morphism is explicitly given by
\begin{equation}
\pi(\pfrak) = \pfrak \cap (A^G)_0 = \pfrak \cap A_0^{G_0}
\end{equation}
for every $\pfrak \in \lvert X \rvert$. In particular, $\pi$ coincides with the classical morphism induced by the inclusion of invariants $A_0^{G_0} \hookrightarrow A_0$.

The expected properties of this construction come together in the following result.

\begin{theorem}
With the notation and assumptions established above, the following properties hold:
\begin{enumerate}
\item the morphism $\pi$ is $G$-invariant;

\item the morphism $\pi$ is affine;

\item if $Z \subseteq X$ is a $G$-invariant closed subset, then $\pi(Z)$ is closed in $X {\sslash} G$;

\item if $Z_1, Z_2 \subseteq X$ are disjoint, $G$-invariant closed subsets, then $\pi(Z_1) \cap \pi(Z_2) = \varnothing$;

\item for every open subset $U \subseteq X {\sslash} G$, there is a canonical isomorphism $\Ocal_{X {\sslash} G}(U) \cong \Ocal_X(\pi^{-1}(U))^G$.
\end{enumerate}
\end{theorem}

\begin{proof}
Properties (2), (3), and (4) follow formally from the classical setting since, at the level of the underlying topological spaces, the morphism $\pi$ coincides with the classical map $\Spec A_0 \to \Spec A_0^{G_0}$. In particular, the behavior of affine open subsets and $G$-invariant closed subsets is identical to that of the classical categorical quotient.

To prove (1), recall that $\pi$ is induced by the algebra inclusion $A^G \hookrightarrow A$; that is, the dual map $\pi^* \colon A^G \to A$ is precisely the inclusion mapping. Let $a \in A^G$. By the definition of invariant elements, we have $\rho(a) = a \otimes 1$. Since $\pi^*(a) = a$, it follows that $\rho(\pi^*(a)) = \pi^*(a) \otimes 1$, which establishes the $G$-invariance of $\pi$.

To prove (5), it suffices to verify the isomorphism on a base of affine open subsets. Consider a basic open set $U = D(f)$ with $f \in (A^G)_0 = A_0^{G_0}$. By definition, we have $\Ocal_{X {\sslash} G}(D(f)) = (A^G)_f$ and $\Ocal_X(\pi^{-1}(D(f))) = \Ocal_X(D(f)) = A_f$. Invoking Lemma~\ref{lem:3.7}, we obtain the natural identification $(A^G)_f \cong (A_f)^G$. This allows us to conclude that $\Ocal_{X {\sslash} G}(D(f)) \cong \Ocal_X(\pi^{-1}(D(f)))^G$, completing the proof.
\end{proof}

A final aspect regarding the behavior of the superquotient under additional finiteness and reductivity assumptions is recorded in the next result.

\begin{proposition}\label{prop:3.9}
Let $X = \sSpec A$ be an affine algebraic superscheme, and let $G$ be a reductive affine algebraic supergroup acting on $X$ via an odd-infinitesimally decoupled coaction $\rho \colon A \to A \otimes_k k[G]$. Then the affine superquotient $X {\sslash} G = \sSpec A^G$ is an affine algebraic superscheme.
\end{proposition}

\begin{proof}
This is an immediate consequence of Proposition~\ref{prop:3.6}.
\end{proof}

The structural result established in the preceding proposition lends itself to a deep geometric interpretation. Under this setup, the classical quotient morphism $\Spec A_0 \to \Spec A_0^{G_0}$ parameterizes the closed orbits of the $G_{0}$-action on $\Spec  A_0$. Since the underlying topological space of the resulting algebraic superscheme $X {\sslash} G$ is precisely this classical spectrum, it follows that the underlying points of this superquotient correspond exactly to the classical closed $G_{0}$-orbits in $\Spec  A_0$.

This characterization indicates that the point-set topology of the quotient remains blind to the odd component of the supergroup. Because the points of $X {\sslash} G$ are entirely dictated by the classical action of $G_{0}$, they fail to detect the additional supergeometric data. Furthermore, the algebraic identity $(A^G)_0=A_0^{G_0}$ from Corollary~\ref{cor:3.5} ensures that no novel invariant functions originate from the odd sectors of the supergroup. Consequently, moving to the supergeometric context does not refine the classification of orbits; the superquotient distinguishes no more orbits than its classical counterpart.

The genuine supergeometric content is instead relegated entirely to the structure sheaf of $X {\sslash} G$. Indeed, because the coordinate superalgebra $A^{G}$ is a non-trivially graded superalgebra, its odd component defines nilpotent directions over the topological base $\Spec A_0^{G_0}$. The algebraic superscheme $X {\sslash} G$ may therefore be interpreted as a ``supergeometric thickening'' of the classical geometric quotient, wherein the additional odd information manifests exclusively at the infinitesimal level through the structure sheaf.

We conclude this discussion with a few illustrative examples.

\begin{example}\label{ex:3.10}
Let us examine an action of the odd multiplicative supergroup $\GG_m^{1 \vert p}$, defined by its coordinate Hopf superalgebra $k[\GG_m^{1 \vert p}] = k[t,t^{-1},\xi_1,\dots,\xi_p]$, on the affine superspace $\AA^{2 \vert p+q}$, whose coordinate superalgebra we write as $k[\AA^{2 \vert p+q}] = k[x,y,\eta_1,\dots,\eta_p,\theta_1,\dots,\theta_q]$.  This action is given by the coaction $\rho \colon k[\AA^{2 \vert p+q}] \to k[\AA^{2 \vert p+q}] \otimes_k k[\GG_{m}^{1 \vert p}]$ specified on generators by
$$
\begin{aligned}
\rho(x) &= x \otimes t, \\
\rho(y) &= y \otimes t, \\
\rho(\eta_i) & = \eta_i \otimes t + x \otimes \xi_i, \\
\rho(\theta_j) &= \theta_j \otimes t.
\end{aligned}
$$
To determine the full superalgebra of invariants under $\rho$, we solve the defining condition $\rho(a) = a \otimes 1$. Proceeding as in Example~\ref{ex:3.2}, a generic homogeneous element of $k[\AA^{2 \vert p+q}]$ is expressed as a linear combination of monomials of the form $x^{i} y^{j} \prod_{\mu=1}^{r}\eta_{i_{\mu}} \prod_{\nu=1}^{s} \theta_{j_{\nu}}$, where $i, j \ge 0$, $1 \le i_1 < \dots < i_r \le p$, and $1 \le j_1 < \dots < j_s \le q$. Evaluating the coaction $\rho$ on each factor of such a monomial yields
$$
\begin{aligned}
\rho(x^{i}) &= x^{i} \otimes t^{i}, \\
\rho(y^{j}) &= y^{j} \otimes t^{j},  \\
 \rho \left( \prod_{\mu=1}^{r} \eta_{i_{\mu}} \right) &=  \prod_{\mu=1}^{r} \eta_{i_{\mu}}  \otimes t^{r}  + \sum_{k=1}^{r} x^{k} \sum_{1 \leq \mu_1 < \cdots < \mu_k \leq r}(-1)^{\sum_{l=1}^{k} (r - k - \mu_l  +l)} \prod_{\substack{\mu=1 \\ \mu \not\in \{ \mu_1,\dots,\mu_k \}}}^{r} \eta_{i_{\mu}}   \otimes t^{r-k}\prod_{l=1}^{k}  \xi_{i_{\mu_{l}}}, \\
\rho \left( \prod_{\nu=1}^{s}\theta_{j_{\nu}}  \right) &=  \prod_{\nu=1}^{s}\theta_{j_{\nu}}   \otimes t^{s}.
\end{aligned}
$$
Multiplying these expressions, the coaction on a generic monomial reads
$$
\begin{aligned}
&x^{i} y^{j}  \prod_{\mu=1}^{r} \eta_{i_{\mu}} \prod_{\nu=1}^{s}\theta_{j_{\nu}}  \otimes t^{i + j + r + s} \\
 &\quad + \sum_{k = 1}^{r} x^{i + k} y^{j} \sum_{1 \leq \mu_1 < \cdots < \mu_k \leq r}(-1)^{\sum_{l=1}^{k} (r - k - \mu_l  +l)} \prod_{\substack{\mu=1 \\ \mu \not\in \{ \mu_1,\dots,\mu_k \}}}^{r} \eta_{i_{\mu}}   \prod_{\nu=1}^{s}\theta_{j_{\nu}}    \otimes t^{i + j+r+s-k} \prod_{l=1}^{k}  \xi_{i_{\mu_{l}}}.
\end{aligned}
$$
The terms in the sum involve non-trivial products of the odd variables $\xi_{i}$. Since the first term contains no such variables, these contributions cannot be canceled. Therefore, the invariance condition compels this sum to vanish, which requires $r = 0$. On the other hand, for the first term to appear in an invariant element, the exponent of $t$ must vanish, which forces $i + j + s = 0$. Since the indices are non-negative, this requires $i=j = s =0$. Hence, the full superalgebra of invariants is
$$
k[\AA^{2\vert p+q}]^{\GG_m^{1 \vert p}} = k.
$$
In line with our general setup, this calculation confirms that the coaction under consideration is odd-infinitesimally decoupled. It also follows at once that the affine superquotient contracts both topologically and structurally to an elemental geometric point:
$$
\AA^{2\vert p+q} {\sslash} \GG_m^{1 \vert p} = \sSpec k.
$$
This structural collapse can be fully understood by examining the underlying orbit geometry of the classical action. On the topological base space, the action reduces to the standard scaling action of $\GG_{m}$ on the classical affine plane $\AA^{2}$. Pictorially, the resulting classical orbits consist of the origin and the geometric rays emanating from it. Since the topological closure of every such ray inevitably contains the origin, these orbits cannot be separated, forcing the entire underlying space, and consequently the affine superquotient itself, to collapse down to a single elemental point.
\end{example}

\begin{example}\label{ex:3.11}
Let us consider the same odd multiplicative supergroup $\GG_{m}^{1 \vert p}$ acting on the same affine superspace $\AA^{2 \vert p+q}$, but now via a different coaction $\rho \colon k[\AA^{2 \vert p+q}] \to k[\AA^{2 \vert p+q}] \otimes_k k[\GG_{m}^{1 \vert p}]$ specified on the generators by
$$
\begin{aligned}
\rho(x) &= x \otimes t, \\
\rho(y) &= y \otimes t^{-1}, \\
\rho(\eta_i) & = \eta_i \otimes t + x \otimes \xi_i, \\
\rho(\theta_j) &= \theta_j \otimes t.
\end{aligned}
$$
The procedure and the resulting calculations are essentially identical to those in the previous example. Taking into account the factor $t^{-1}$ in the coaction on $y$, evaluating $\rho$ on a generic monomial $x^i y^j \prod_{\mu=1}^r \eta_{i_\mu} \prod_{\nu=1}^s \theta_{j_\nu}$ yields
$$
\begin{aligned}
&x^{i} y^{j}  \prod_{\mu=1}^{r} \eta_{i_{\mu}} \prod_{\nu=1}^{s}\theta_{j_{\nu}}  \otimes t^{i - j + r + s} \\
 &\quad + \sum_{k = 1}^{r} x^{i + k} y^{j} \sum_{1 \leq \mu_1 < \cdots < \mu_k \leq r}(-1)^{\sum_{l=1}^{k} (r - k - \mu_l  +l)} \prod_{\substack{\mu=1 \\ \mu \not\in \{ \mu_1,\dots,\mu_k \}}}^{r} \eta_{i_{\mu}}   \prod_{\nu=1}^{s}\theta_{j_{\nu}}    \otimes t^{i - j+r+s-k} \prod_{l=1}^{k}  \xi_{i_{\mu_{l}}}.
\end{aligned}
$$
As before, the non-trivial presence of the odd variables $\xi_{i}$ in the sum forces this contribution to vanish, which requires $r = 0$. On the other hand, for the first term to be invariant, the net exponent of $t$ must vanish identically, which dictates $i - j + s = 0$, or equivalently $i + s =j$. The monomials satisfying the invariance condition are therefore of the form
$$
x^{i} y^{i + s}\theta_{j_1} \cdots \theta_{j_s} = (xy)^{i} ( y \theta_{j_1}) \cdots ( y \theta_{j_s}). 
$$ 
It follows that the even part of the invariant superalgebra is generated entirely by the classical invariant $u = xy$, while its odd part is generated over $k[u]$ by the odd monomials $\zeta_j = y \theta_j$ for $j = 1, \dots, q$. Because each $\theta_j$ is an odd coordinate function, these generators satisfy the nilpotency relations $\zeta_j^2 = (y\theta_j)^2 = y^2\theta_j^2 = 0$. Since no further independent algebraic relations are imposed, the full superalgebra of invariants is
$$
k[\AA^{2\vert p+q}]^{\GG_m^{1 \vert p}} = k[u,\zeta_1,\dots,\zeta_q]. 
$$
This final algebraic identity confirms that the coaction under consideration is again odd-infinitesimally decoupled. It likewise establishes that the affine superquotient survives as a genuine affine algebraic superscheme, namely
$$
\AA^{2\vert p+q} {\sslash} \GG_m^{1 \vert p} = \sSpec  k[u,\zeta_1,\dots,\zeta_q] = \AA^{1 \vert q}. 
$$
Notice that this result provides a vivid illustration of the theoretical framework established above, explicitly demonstrating how the point-set topology remains blind to the odd variables of the supergroup. Under the reduced coaction, the action corresponds to the standard hyperbolic scaling of $\GG_{m}$ on the classical affine plane $\mathbb{A}^{2}$. Pictorially, the resulting orbits consist of the origin, the punctured coordinate axes, and the families of hyperbolas defined by $xy = \text{const}$. Since the topological closures of both punctured axes inevitably contain the origin, these specific orbits cannot be separated in the categorical framework and collapse into a single point. Conversely, the hyperbolas remain closed and topologically separated from one another, allowing the underlying topological space of the superquotient to behave as a well-behaved, separable classical affine line where the coordinate $u = xy$ cleanly parameterizes the closed orbits. The odd structure of the coaction leaves this classical topological base intact, while endowing the superquotient with the $q$ independent odd directions $\zeta_j = y\theta_j$ that survive in the invariant superalgebra.
\end{example}

\begin{example}\label{ex:3.12}
Let us consider once again that the same odd multiplicative supergroup $\GG_{m}^{1 \vert p}$ acts on the same affine superspace $\AA^{2 \vert p+q}$, but under a modified coaction $\rho \colon k[\AA^{2 \vert p+q}] \to k[\AA^{2 \vert p+q}] \otimes_k k[\GG_{m}^{1 \vert p}]$ designed to induce a fundamental structural imbalance between the even and odd variables. On the generators of the coordinate superalgebra, this action is prescribed by
$$
\begin{aligned}
\rho(x) &= x \otimes t, \\
\rho(y) &= y \otimes t, \\
\rho(\eta_i) & = \eta_i \otimes t^{-1} + x \otimes \xi_i, \\
\rho(\theta_j) &= \theta_j \otimes t^{-1}.
\end{aligned}
$$
As in the preceding two examples, the computation proceeds in exactly the same way, now accounting for the inverse dependence on $t$ in the coactions on the odd generators. Evaluating $\rho$ on a generic monomial $x^{i} y^{j} \prod_{\mu=1}^{r} \eta_{i_{\mu}} \prod_{\nu=1}^{s} \theta_{j_{\nu}}$ yields
$$
\begin{aligned}
&x^{i} y^{j}  \prod_{\mu=1}^{r} \eta_{i_{\mu}} \prod_{\nu=1}^{s}\theta_{j_{\nu}}  \otimes t^{i + j - r - s} \\
 &\quad + \sum_{k = 1}^{r} x^{i + k} y^{j} \sum_{1 \leq \mu_1 < \cdots < \mu_k \leq r}(-1)^{\sum_{l=1}^{k} (r - k - \mu_l  +l)} \prod_{\substack{\mu=1 \\ \mu \not\in \{ \mu_1,\dots,\mu_k \}}}^{r} \eta_{i_{\mu}}   \prod_{\nu=1}^{s}\theta_{j_{\nu}}    \otimes t^{i + j - r - s + k} \prod_{l=1}^{k}  \xi_{i_{\mu_{l}}}.
\end{aligned}
$$
Once again, the presence of the independent odd variables $\xi_{i}$ in the summation forces the second component to vanish, which immediately implies $r = 0$. For the remaining term to achieve invariance, the overall power of $t$ must equal zero, requiring $i + j - s = 0$, or equivalently $j = s - i$. Consequently, the invariant monomials are precisely those given by
$$
x^{i} y^{s - i}  \theta_{j_1} \cdots \theta_{j_s} = (x \theta_{j_1}) \cdots (x \theta_{j_i}) (y \theta_{j_{i+1}}) \cdots (y \theta_{j_1}). 
$$
This factorization reveals that the even part of the invariant superalgebra reduces entirely to the scalar field $k$, leaving the full superalgebra generated over $k$ by the odd monomials $\psi_j = x \theta_j$ and $\zeta_j = y \theta_j$ for $j = 1, \dots, q$. The superalgebra of invariants is therefore given by
$$
k[\AA^{2\vert p+q}]^{\GG_m^{1 \vert p}} = k[\psi_1,\dots,\psi_q,\zeta_1,\dots,\zeta_q]. 
$$
From this final algebraic identity, it is readily verified that the coaction under consideration is odd-infinitesimally decoupled. Moreover, unlike the situation in the preceding example, the resulting affine superquotient in this case collapses to a purely odd superspace, namely
$$
\AA^{2\vert p+q} {\sslash} \GG_m^{1 \vert p} = \sSpec k[\psi_1,\dots,\psi_q,\zeta_1,\dots,\zeta_q] = \AA^{0 \vert 2q}. 
$$
It is worth emphasizing the remarkable nature of this outcome: starting from the affine superspace $\AA^{2\vert p+q}$ with a non-trivial even core, the affine superquotient construction yields a purely odd superspace $\AA^{0 \vert 2q}$. This illustrates that algebraic coactions of supergroups can entirely annihilate the classical reduced space of the superquotient, reducing its underlying topological space to a single point while preserving a non-trivial odd structure.
\end{example}


\section{Character coactions and GIT superquotients}\label{sec:4}
This section contains the main results of our work, which center on extending the classic construction of GIT quotients from Section 2 of King's foundational paper \cite{King1994} to the supergeometric framework, thereby allowing us to define and construct GIT superquotients of affine superschemes. We first generalize the definition of relative invariants associated with a coaction of an affine algebraic supergroup on a commutative superalgebra and a choice of a character to the super-setting, emphasizing the key similarities and differences with the classical situation. Following this development, we adapt King's recipe to define the notions of semi-stability and stability in the context of affine algebraic supergroup actions on affine superschemes. We then characterize these stability conditions in terms of a numerical criterion in close analogy with King's paper, which ultimately allows us to construct the GIT superquotient of such superschemes.

\subsection{Relative invariants associated with a character}\label{sec:4.1}
We start by introducing the notion of a relative invariant associated with a character, within the context of a supergroup coaction on a commutative superalgebra. All notation and terminology established in \S~\ref{sec:3.1} remain in force.

Let $A$ be a commutative superalgebra equipped with a coaction $\rho \colon A \to A \otimes_k k[G]$ of an affine algebraic supergroup $G$. Fix a character $\chi^{*} \colon k[\GG_m] \to k[G]$, and let $\chi \in k[G]_0$ denote its associated group-like element. An element $a \in A$ is said to be a \emph{relative invariant of weight} $\chi$ if
\begin{equation}
\rho(a) = a \otimes \chi.
\end{equation}
The set of such relative invariants forms a super vector subspace of $A$, denoted by $A^{G,\chi }$ and explicitly defined as
\begin{equation}
A^{G,\chi} = \{ a \in A \mid  \rho(a) = a \otimes \chi \}. 
\end{equation}
This super vector subspace inherits the natural $\ZZ_{2}$-grading decomposition $A^{G,\chi} = (A^{G,\chi})_0 \oplus (A^{G,\chi})_1$, where the even and odd sectors are given by $(A^{G,\chi})_0 = A^{G,\chi} \cap A_0$ and $(A^{G,\chi})_1 = A^{G,\chi} \cap A_1$, respectively.

This construction naturally extends to the integral powers of the character. For each $n \in \ZZ$, the element $\chi ^{n}$ is likewise group-like and thus defines a character $(\chi^n)^{*} \colon k[\GG_m] \to k[G]$. Consequently, one can consider the collection of super vector subspaces
\begin{equation}
A^{G,\chi^{n}} = \{ a \in A \mid  \rho(a) = a \otimes \chi^{n} \}. 
\end{equation}
Such a collection is compatible with the multiplicative structure of $A$ in the sense that, if $a \in A^{G,\chi^m}$ and $b \in A^{G,\chi^n}$, then the product of these elements satisfies $ab \in A^{G,\chi^{m+n}}$. Therefore, the direct sum $\bigoplus_{n \in \ZZ} A^{G,\chi^n}$ inherits a natural $\ZZ_{\geq 0}$-graded superalgebra structure. This will play a role shortly.

To streamline what follows, and recalling that $r_0 \colon k[G] \to k[G_0]$ is the canonical restriction morphism, we shall adopt the shorthand notation $\chi_0 = r_0(\chi)$ for the classical group-like element corresponding to any $\chi \in k[G]_0^{\times}$. In connection with the foregoing, we record the following fact.

\begin{lemma}\label{lem:4.1}
Let $G$ be an affine algebraic supergroup and $G_{0}$ its underlying affine algebraic group. For any coaction $\rho \colon A \to A \otimes_k k[G]$ on a commutative superalgebra, and any character $\chi^{*} \colon k[\GG_m] \to k[G]$ with associated group-like element $\chi \in k[G]_0^{\times}$, one has $A^{G,\chi } \subseteq A^{G_0,\chi_0}$. 
\end{lemma}

\begin{proof}
The proof is essentially identical to that of Lemma~\ref{lem:3.1}. Specifically, one shows by exactly the same calculation that if an element is a relative invariant of weight $\chi$ with respect to the full coaction $\rho$, it will likewise be a relative invariant of weight $\chi_0$ with respect to the reduced coaction $\rho_0 = (\text{id}_A \otimes r_0) \circ \rho$.
\end{proof}

Just as with the superalgebra of invariants, one should not expect the reverse inclusion of Lemma~\ref{lem:4.1} to hold true in general. Indeed, in complete analogy with that absolute setting, there is a generalization of the decoupling property established in Proposition~\ref{prop:3.4}.

\begin{proposition}\label{prop:4.2}
Let $G$ be an affine algebraic supergroup and $(G_0,\gfrak)$ its corresponding super Harish--Chandra pair. For any coaction $\rho \colon A \to A \otimes_k k[G]$ on a commutative superalgebra $A$, and any character $\chi^{*} \colon k[\GG_m] \to k[G]$ with associated group-like element $\chi \in k[G]_0^{\times}$, one has 
\begin{equation}
A^{G,\chi } = A^{G_0,\chi_0} \cap A^{\gfrak_1}. 
\end{equation}
\end{proposition}

\begin{proof}
The proof follows the same guidelines as that of Proposition~\ref{prop:3.4}, with only minor modifications. To wit, any odd $\varepsilon_G$-superderivation $D \in \gfrak_1$ is, by definition, an odd linear map $D \colon k[G] \to k$. Since $k$ is concentrated in degree zero, $D$ necessarily vanishes on $k[G]_0$. Consequently, as $\chi$ is an even element, $D(\chi) = 0$ for any such $D \in \gfrak_1$. The rest of the proof of Proposition~\ref{prop:3.4} then applies verbatim.
\end{proof}

Naturally, this result has as a consequence the following analogue of Corollary~\ref{cor:3.5}, which provides a description of the even component of the relative invariants.

\begin{corollary}\label{cor:4.3}
With the notation and hypotheses as above, one has
\begin{equation}
(A^{G,\chi})_0 = A_0^{G_0,\chi} \cap A_0^{\gfrak_1}.
\end{equation}
In particular, we obtain $(A^{G,\chi})_0 = A_0^{G_0,\chi_0}$ if and only if the coaction $\rho$ is odd-infinitesimally decoupled.
\end{corollary}

\begin{proof}
The proof carries over verbatim from that of Corollary~\ref{cor:3.5}, since the statement is obtained by taking the even component of the equality established in Proposition~\ref{prop:4.2}.
\end{proof}

In order to move forward, we must develop our vocabulary. A $\ZZ_{\ge 0}$-graded commutative superalgebra $R = \bigoplus_{n \geq 0} R_{n}$ is said to be of \emph{finite type} if $R_{0}$ is a commutative superalgebra of finite type and $R$ is finitely generated as an algebra over $R_{0}$ by a finite set of elements sitting in degree one that are homogeneous with respect to the $\ZZ_{2}$-grading. Under these assumptions, each graded component $R_{n}$ is a finite-dimensional super vector space.

With this in place, we have the following result.

\begin{proposition}\label{prop:4.4}
Let $A$ be a commutative superalgebra of finite type, let $\rho \colon A \to A \otimes_k k[G]$ be an odd-infinitesimally decoupled coaction of a reductive affine algebraic supergroup $G$, and let $\chi^{*} \colon k[\GG_m] \to k[G]$ be a character with associated group-like element $\chi \in k[G]_0^{\times}$. Then there exists a positive integer $N \geq 1$ such that the $\mathbb{Z}_{\ge 0}$-graded commutative superalgebra $\bigoplus_{n\geq 0} A^{G,(\chi^{N})^{n}}$ is of finite type.
\end{proposition}

\begin{proof}
We set $R = \bigoplus_{n \geq 0} A^{G, \chi^n}$ and begin by analyzing its even part, which expands as the direct sum $R_0 = \bigoplus_{n \geq 0} (A^{G, \chi^n})_0$. Under the hypothesis that the coaction $\rho$ is odd-infinitesimally decoupled, Corollary~\ref{cor:4.3} establishes that the even relative invariants under the supergroup match the classical ones at every weight level, forcing $(A^{G,\chi^n})_0 = A_0^{G_0,\chi_0^n}$. This results in the identification $R_{0}=\bigoplus _{n\ge 0}A_{0}^{G_{0},\chi _{0}^{n}}$.  The finite type property of $A$ ensures that its even component $A_0$ is a finitely generated algebra. By applying the classical Hilbert--Nagata theorem for graded algebras of semi-invariants under reductive group actions, it follows that $R_0$ is a finitely generated algebra, making it a Noetherian ring.

Next, we inspect the odd part, which is given by the direct sum $R_1 = \bigoplus_{n \geq 0} (A^{G, \chi^n})_1$. We compare this with the auxiliary graded $R_0$-module of classical odd semi-invariants defined by $\tilde{R}_{1}=\bigoplus _{n\ge 0}(A^{G_{0},\chi _{0}^{n}})_{1}$. Because $A$ is of finite type, $A_1$ is a finitely generated $A_0$-module. The linear reductivity of $G_0$ means we can apply Noether’s classical theorem for modules under reductive group actions, which guarantees that $\tilde{R}_{1}$ is a finitely generated $R_0$-module. The Noetherian nature of $R_0$ then implies that $\tilde{R}_{1}$ is a Noetherian $R_0$-module. Meanwhile, Lemma~\ref{lem:4.1} establishes the term-by-term containment $(A^{G,\chi^n})_1 \subseteq (A^{G_0,\chi_0^n})_1$ for all $n \geq 0$, which translates globally into the $R_0$-module inclusion $R_{1}\subseteq \tilde{R}_{1}$. As an $R_0$-submodule of the Noetherian module $\tilde{R}_{1}$, the component $R_1$ inherits the property of being a finitely generated $R_0$-module. 

As a result of this analysis, the full graded superalgebra $\bigoplus_{n \geq 0} A^{G,\chi^n}$ is generated over $A^{G}$ by a finite set of homogeneous elements that may sit in various higher degrees of the $\ZZ_{\geq 0}$-grading. By finding a common multiple for the degrees of these finite generators, the standard Veronese truncation for graded superalgebras allows us to re-index the grading. Thus, there exists a positive integer $N \geq 1$ such that the Veronese subalgebra $\bigoplus_{n \geq 0} A^{G, (\chi^{N})^{n}}$ is generated as an algebra over its degree-zero component $A^{G}$ by a finite set of $\ZZ_{2}$-homogeneous elements sitting strictly in degree one. Since $A^{G}$ is of finite type by Proposition~\ref{prop:3.6}, this confirms that $\bigoplus_{n \geq 0} A^{G, (\chi^{N})^{n}}$ is of finite type, completing the proof.
\end{proof}

Retaining our setup where the commutative superalgebra $A$ is equipped with an odd-infinitesimally decoupled coaction $\rho \colon A \to A \otimes_k k[G]$ of an affine algebraic supergroup $G$, let us further assume that there exists a closed normal supersubgroup $\Delta \subseteq G$, corresponding to a Hopf superideal $I_{\Delta} \subseteq k[G]$, such that the induced coaction $\rho_{\Delta} = (\id_A \otimes r_{\Delta}) \circ \rho$ is trivial, where $r_{\Delta} \colon k[G] \to k[G]/I_{\Delta} = k[\Delta]$ denotes the quotient morphism. The following result holds.

\begin{lemma}\label{lem:4.5}
Under the additional hypotheses above, let $\Delta_0 \subseteq G_0$ be the underlying closed subgroup corresponding to $\Delta \subseteq G$. Then, for any character $\chi^{*} \colon k[\mathbb{G}_m] \to k[G]$ with associated group-like element $\chi \in k[G]_{0}^{\times}$, a sufficient condition for $\chi_0(\Delta_0) = \{1\}$ is that $r_{\Delta}(\chi) = 1$.
\end{lemma}

\begin{proof}
Let $\bar{r}_0 \colon k[\Delta] \to k[\Delta_0]$ be the canonical reduction map for $\Delta$, and let $r_{\Delta_0} \colon k[G_0] \to k[\Delta_0]$ be the quotient map induced by the inclusion $\Delta_0 \subseteq G_0$. By the functoriality of the reduction mapping, we have the commutativity relation $r_{\Delta_0}(r_0(\chi)) = \bar{r}_0(r_{\Delta}(\chi))$. Since the classical character is given by $\chi_0 = r_0(\chi)$, its restriction to $\Delta _{0}$ corresponds to $r_{\Delta_0}(\chi_0)$. Assuming $r_{\Delta}(\chi) = 1$ holds, we immediately obtain
$$
\chi_0 \vert \Delta_0 = r_{\Delta_{0}}(\chi _{0})=\bar{r}_{0}(r_{\Delta }(\chi ))=\bar{r}_{0}(1)=1,
$$
which proves that $\chi_0(\Delta_0) = \{1\}$.
\end{proof}

The sufficiency condition in Lemma~\ref{lem:4.5} is not necessary in general due to the presence of non-trivial infinitesimal structures. Indeed, the requirement $\chi_0(\Delta_0) = \{1\}$ only implies that the projection $\bar{\chi} = r_{\Delta}(\chi)$ belongs to $1 + (k[\Delta]_1)_0$, where $(k[\Delta]_1)_0$ denotes the even part of the ideal generated by the odd elements of $k[\Delta]$. If there exists a non-zero element $\zeta \in (k[\Delta]_1)_0$ whose coproduct gives $\zeta \otimes 1 + 1 \otimes \zeta + \zeta \otimes \zeta$, then $\bar{\chi} = 1 + \zeta$ defines a valid group-like element in $k[\Delta]$ which is not identically $1$.
\subsection{GIT superquotients of affine superschemes}
We now generalize the framework of affine superquotients established in \S\ref{sec:3.2} to implement King's GIT recipe for affine algebraic superschemes. We shall retain the geometric and algebraic setting from both that earlier discussion and the previous paragraph.

Let $X = \sSpec A$ be an affine superscheme, and let $G$ be an affine algebraic supergroup acting on $X$ by means of an odd-infinitesimally decoupled coaction $\rho \colon A \to A \otimes_k k[G]$. Alongside this latter assumption, we assume the existence of a closed normal subsupergroup $\Delta \subseteq G$ with a trivial induced coaction and with underlying closed normal subgroup $\Delta_0 \subseteq G_0$.

To prepare what follows, we find it convenient to denote the points of the underlying topological space $\lvert X \rvert$ by $x \in \lvert X \rvert$, designating by $\pfrak_x \subseteq A_0$ the corresponding prime ideal under the identification $\lvert X \rvert = \Spec A_0$. Likewise, we write $k(x) = (A_0)_{\pfrak_x} / \pfrak_x (A_0)_{\pfrak_x}$ for the residue field of $x$, and denote by $\ev_x \colon A_0 \to k(x)$ the corresponding evaluation homomorphism, defined as the composition of the canonical localization $A_0 \to (A_0)_{\pfrak_x}$ with the projection $(A_0)_{\pfrak_x} \to k(x)$. For an element $f \in A_0$, its evaluation at $x$ is given by $f(x) = \ev_x(f)$, which is non-zero precisely when $f \notin \pfrak_x$.

Now comes the key definition. Fix a character $\chi^{*} \colon k[\mathbb{G}_m] \to k[G]$, and let $\chi \in k[G]_0^{\times}$ be its associated group-like element. A point $x \in \lvert X \rvert$ is said to be $\chi$-\emph{semistable} if there exist an integer $n \geq 1$ and a relative invariant $f \in (A^{G,\chi^n})_0$ such that $f(x) \neq 0$. Furthermore, the point is called $\chi$-\emph{stable} if, in addition to being $\chi$-semistable, its orbit $G_0 \cdot x$ under the underlying affine algebraic group $G_{0}$ is a closed subset of the basic open subset $D(f)$ and satisfies the dimensional identity $\dim G_0 \cdot x = \dim G_0 / \Delta_0$.

The choice of restricting the definition of $\chi$-semistability to the even component $(A^{G,\chi^n})_0$ is justified by our active hypothesis that the coaction is odd-infinitesimally decoupled. Under this condition, Corollary~\ref{cor:4.3} guarantees the structural identification $(A^{G,\chi^n})_0 = A_0^{G_0,\chi_0^n}$, ensuring that any relative invariant used to test stability is already concentrated in the even component and evaluates directly at any point $x$ as an ordinary classical relative invariant of weight $\chi _{0}^{n}$ under the underlying affine algebraic group $G_{0}$. This same topological behavior rules the dimensional constraint in the definition of $\chi$-stability. Because the underlying space registers only classical commutative data, any purely odd or nilpotent components of the closed normal subsupergroup $\Delta$ vanish upon passing to the prime spectrum. The infinitesimal geometry of the odd variables is therefore blind to the dimension of the topological support of the orbit, forcing the dimension of $G_0 \cdot x$ to depend exclusively on the underlying closed normal subgroup $\Delta _{0}$.

Having clarified this definition, and following King's classical construction of GIT quotients, we define the \emph{GIT superquotient} of $X$ by $G$ with respect to $\chi$ as the superscheme
\begin{equation}
X {\sslash_{\chi}} G = \sProj \left( \bigoplus_{n \geq 0} A^{G,\chi^{n}}\right). 
\end{equation}
Although the $\sProj$ construction was not reviewed in Section~\ref{sec:2}, it can nevertheless be described explicitly in close analogy with the construction of the spectrum (see, e.g., \cite{BruzzoHernandezRuiperezPolishchuk2023}). Indeed, the underlying topological space is a priori defined as the projective spectrum of $\bigoplus_{n \geq 0} (A^{G,\chi^{n}})_0$. Our active odd-infinitesimally decoupling hypothesis, however, forces this algebra to coincide with $\bigoplus_{n \geq 0} A_0^{G_0,\chi_0^{n}}$, so that
\begin{equation}
\lvert X {\sslash_{\chi}} G \rvert = \Proj \left( \bigoplus_{n \geq 0} A_0^{G_0,\chi_0^{n}}\right). 
\end{equation}
As such, this space is endowed with the Zariski topology whose basic open subsets are of the form $D_+(f)$ for a homogeneous element $f \in A_0^{G_0,\chi_0^{n}}$ of positive degree $n \geq 0$. The structure sheaf $\Ocal_{X {\sslash_{\chi}} G}$ is the sheaf of commutative superalgebras whose sections over each such neighborhood are given by the commutative superalgebra
\begin{equation}
\Ocal_{X {\sslash_{\chi}} G}(D_+(f)) = \left\{ \frac{a}{f^k} \;\bigg\vert{}\; k \in \ZZ_{\geq 0},\ a \in A^{G,\chi^{kn}} \right\},
\end{equation}
under which the $\ZZ_{2}$-grading is naturally inherited from that of $A$. Equivalently, these sections are obtained precisely as the degree-zero component with respect to the $\ZZ$-grading of the homogeneous localization of the full graded superalgebra $\bigoplus_{n \geq 0} A^{G,\chi^{n}}$.

The following result characterizes the behavior of the GIT superquotient under additional finiteness and reductivity assumptions.

\begin{proposition}\label{prop:4.6}
Let $X = \sSpec A$ be an affine algebraic superscheme, let $G$ be a reductive affine algebraic supergroup acting on $X$ via an odd-infinitesimally decoupled coaction $\rho \colon A \to A \otimes_k k[G]$, and let $\chi^{*} \colon k[\GG_m] \to k[G]$ be a character with associated group-like element $\chi \in k[G]_0^{\times}$. Then the GIT superquotient $X {\sslash_{\chi}} G$ is a projective superscheme over the affine superquotient $X {\sslash} G$. 
\end{proposition}

\begin{proof}
For the precise definition of a projective superscheme over a base, we refer the reader to \S2.5 of \cite{BruzzoHernandezRuiperezPolishchuk2023}. The proof proceeds by verifying the conditions required by this definition. By the general $\sProj$ construction, the canonical inclusion of superalgebras $A^G \hookrightarrow \bigoplus_{n \geq 0} A^{G,\chi^{n}}$ onto the degree-zero component induces the natural morphism of superschemes $\varphi \colon X {\sslash_{\chi}} G \to X {\sslash} G$. On the other hand, Proposition~\ref{prop:4.4} guarantees the existence of a positive integer $N \geq 1$ such that the $\ZZ_{\ge 0}$-graded commutative superalgebra $\bigoplus_{n \geq 0} A^{G,(\chi^{N})^{n}}$ is of finite type. By appealing again to the general properties of the $\sProj$ construction, this time under Veronese re-indexing, there is a canonical isomorphism of superschemes $X {\sslash_{\chi}} G \cong \sProj(\bigoplus_{n \geq 0} A^{G,(\chi^{N})^{n}})$ that is compatible with the structural morphism $\varphi$.  Now, since $\bigoplus_{n \geq 0} A^{G,(\chi^{N})^{n}}$ is generated as an algebra over its degree-zero component $A^{G}$ by a finite set of elements sitting in degree one that are homogeneous with respect to the $\ZZ_{2}$-grading, this identification immediately guarantees the existence of a closed immersion of superschemes $i \colon X {\sslash_{\chi}} G \hookrightarrow \PP^{m \vert n}_{A^G}$, where $\PP_{A^{G}}^{m \vert n}$ denotes the projective superspace of dimension $m \vert n$ over the commutative superalgebra $A^{G}$, with $m+1$ and $n$ corresponding to the number of even and odd generators of degree one, respectively. This embedding is compatible with the structural maps in the sense that $p \circ i = \varphi$, where $p \colon \PP^{m \vert n}_{A^G} \to X {\sslash} G$ is the canonical projection. Thus, the morphism $\varphi$ factors through a projective superspace, proving that the GIT superquotient $X {\sslash_{\chi}} G$ is a projective superscheme over the affine superquotient $X {\sslash} G$.
\end{proof}

The projectivity of the GIT superquotient established in the preceding proposition extends the geometric philosophy previously observed for the affine superquotient. At the topological level, the structural morphism induces directly the classical projective morphism of schemes $\Proj (\bigoplus_{n \geq 0} A_0^{G_0, \chi_0^n}) \to \Spec A_0^{G_0}$. This implies that the underlying topological spaces and their point-set relations are entirely determined by the classical relative invariants of the group $G_{0}$. Consequently, the underlying topological morphism remains unaffected by the odd sectors of the supergroup action. Instead, the genuinely supergeometric features of this GIT superquotient are, in perfect analogy with the affine setting, contained entirely within the structure sheaves of the superschemes.

We assume from this point forward that the hypotheses of Proposition~\ref{prop:4.6} hold. Within this setting, much as in the classical case, the central point of geometric invariant theory is that the GIT superquotient $X {\sslash_{\chi}} G$ admits a more geometric description. For this purpose, we return to the $\chi$-semistable points of $X$ introduced earlier. Let us denote the subset of such points by $X_{\chi }^{\uss}$. It is clear that this is an open subset of $X$ that is also $G$-invariant, meaning that it is stable under the action of the affine algebraic group $G_{0}$ on the underlying topological space $\lvert X \rvert$. We define an equivalence relation on these points by declaring that $x \sim y$ if and only if the orbit closures $\overline{G_0 \cdot x}$ and $\overline{G_0 \cdot y}$ intersect in $X_{\chi }^{\uss}$. We shall call two $\chi$-semistable points or orbits identified by this relation ``GIT equivalent''. The following result gives this relation a precise geometric meaning.

\begin{proposition}\label{prop:4.7}
There exists a $G$-invariant morphism $\pi^{\uss}_{\chi}\colon X^{\text{ss}}_{\chi} \to X \sslash_{\chi} G$ fitting into the commutative diagram
\begin{equation}
\begin{tikzcd}[row sep=3.5em,  
column sep=4.0em,
  every label/.append style={font=\normalsize}
]
X^{\uss}_{\chi} \arrow[r, -{To[length=2.5pt, width=4pt]}, hook]  \arrow[d, -{To[length=2.5pt, width=4pt]}, "\pi^{\uss}_{\chi}"'{yshift=3.5pt}]  & X   \arrow[d, -{To[length=2.5pt, width=4pt]}, "\pi"] \\
X {\sslash_{\chi}} G \arrow[r, -{To[length=2.5pt, width=4pt]}, "\varphi"] & X {\sslash} G
\end{tikzcd}
\end{equation}
where the top horizontal arrow is the canonical open immersion and, as before, $\pi$ stands for the affine superquotient morphism. Furthermore, two points $x, y \in X^{\text{ss}}_{\chi}$ lie in the same fiber of $\pi^{\uss}_{\chi}$ if and only if they are GIT equivalent.
\end{proposition}

\begin{proof}
We construct the morphism $\pi^{\uss}_{\chi}$ locally, using once again the finite generation established in Proposition~\ref{prop:4.4}, which guarantees the existence of a positive integer $N \geq 1$ such that the graded superalgebra $\bigoplus_{n \geq 0} A^{G,(\chi^{N})^{n}}$  is generated as an algebra over its degree-zero component $A^{G}$ by finitely many homogeneous elements of degree one, with respect to the $\ZZ_2$-grading. Let $f \in A_0^{G_0,\chi_0^N}$ be an even generator from this finite set. The corresponding basic open subset $D(f)$ is then an affine superscheme given by $D(f) = \sSpec A_f$. These affine open subsuperschemes collectively cover the set of $\chi$-semistable points $X_{\chi }^{\uss}$. Under the coaction $\rho$, the fact that $f$ is a relative invariant means that the localized superalgebra $A_{f}$ inherits a canonical coaction of $G$. The superalgebra of invariants $A_{f}^{G}$ defines the local coordinate superalgebra of the distinguished affine open subsuperschemes $D_+(f)$ that cover the GIT superquotient $X {\sslash_{\chi}} G$. The canonical inclusions of these superalgebras of invariants $A_f^G \hookrightarrow A_f$ induce local morphisms of affine superschemes from $D(f)$ to $D_+(f)$. By the universal properties of the $\sProj$ construction and the compatibility of localizations (cf.~\cite[\S2.3]{BruzzoHernandezRuiperezPolishchuk2023}), these local maps glue together uniquely to define the global $G$-invariant morphism of superschemes $\pi^{\uss}_{\chi} \colon X^{\uss}_{\chi} \to X {\sslash_{\chi}} G$. On each basic open subset $D(f)$, the restriction of the affine superquotient morphism $\pi$ is given by the composition of the structural map $A^G \to A$ with the canonical localization $A \to A_f$. Since the image of $A^{G}$ consists of invariant elements, this composition factors through the superalgebra of invariants via the map $A^G \to A_f^G$. This local algebraic factorization matches the restriction of the structural morphism $\varphi$ to $D(f)$, which guarantees that the diagram commutes globally.

We now turn to the second assertion of the proposition concerning the fibers of the morphism $\pi^{\uss}_{\chi}$. Let $x$ and $y$ be two $\chi$-semistable points in $X_{\chi }^{\uss}$. By definition, if these points are GIT equivalent, their orbit closures $\overline{G_{0}\cdot x}$ and $\overline{G_{0}\cdot y}$ intersect in $X_{\chi }^{\uss}$. Since the local sections of $\Ocal_{X {\sslash_{\chi}} G}$ are given over each distinguished open subset $D_+(f)$ by fractions of the form $a/f^k$ where $a \in A^{G,\chi^{kn}}$, the character weights under the action of the underlying affine algebraic group $G_{0}$ cancel out completely. Consequently, these local sections are constant along the closure of any $G_{0}$-orbit. Any such section evaluating to a value at $x$ must take the same value at any point in the intersection $\overline{G_{0}\cdot x} \cap \overline{G_{0}\cdot y}$, and hence must coincide with its value at $y$. Because their images cannot be separated by any local sections of $X {\sslash_{\chi}} G$, they must map to the exact same point under $\pi^{\uss}_{\chi}$, which implies they lie in the same fiber.

Conversely, suppose that $x$ and $y$ lie in the same fiber of $\pi^{\uss}_{\chi}$. This implies their images fall into a common distinguished affine open subset $D_+(f)$ of  $X {\sslash_{\chi}} G$. Consequently, $x$ and $y$ must both belong to the corresponding basic open subset $D(f)$ that covers $X_{\chi }^{\uss}$. Topologically, the underlying space of this basic open subset coincides with that of the classical affine scheme $\Spec (A_0)_f$, where the supergroup action reduces entirely to the classical action of the underlying affine algebraic group $G_{0}$. By the classical theory of reductive group actions on affine schemes, the closure of any $G_{0}$-orbit within this locus contains a unique closed $G_{0}$-orbit. Since $x$ and $y$ share the same evaluations under all local sections of $\Ocal_{X {\sslash_{\chi}} G}$, it follows that they cannot be separated by the ordinary local sections of $\Ocal_{\Spec A_0 {\sslash_{\chi_0}} G_{0}}$. According to classical affine GIT, this forces their orbit closures to contain the same unique closed $G_0$-orbit. The existence of this common closed orbit implies that the orbit closures $\overline{G_{0}\cdot x}$ and $\overline{G_{0}\cdot y}$ intersect inside $X_{\chi }^{\uss}$, meaning that they are, by definition, GIT equivalent. This completes the proof.
\end{proof}

We shall refer to the morphism $\pi^{\uss}_{\chi} \colon X^{\uss}_{\chi} \to X {\sslash_{\chi}} G$ from Proposition~\ref{prop:4.7} as the \emph{GIT superquotient morphism}. It should perhaps be noticed that when the group-like element $\chi$ is trivial, as is the case in the ordinary setting, $\chi$-semistability becomes a vacuous condition, meaning that the GIT superquotient $X {\sslash_{\chi}} G$ reduces precisely to the affine superquotient $X {\sslash} G$. Consequently, the GIT superquotient morphism $\pi^{\uss}_{\chi}$ coincides directly with the affine superquotient morphism $\pi$.

There is also a more geometric characterization of semistability and stability, formulated in terms of the coaction of $G$ on a polynomial extension of the coordinate superalgebra $A$ of $X$. This is the coaction-theoretic analogue of the lifting of the $G$-action to the ``total space of the line bundle'' associated with the group-like element $\chi$, as discussed in King's paper. Specifically, we consider the polynomial superalgebra $A[z]$, where $z$ is an even variable, equipped with the lifted coaction $\tilde{\rho} \colon A[z] \to A[z] \otimes_k k[G]$ defined by
\begin{equation}
\begin{aligned}
\tilde{\rho}(a)&=\rho (a), \\
\tilde{\rho}(z)&=z\otimes \chi ^{-1},
\end{aligned}
\end{equation}
and extended as a morphism of superalgebras. A point $x \in \lvert X \rvert$, corresponding to a prime ideal $\pfrak_x \subseteq A_0$, admits liftings $\tilde{x} \in \lvert \sSpec A[z] \rvert$. These liftings correspond to prime ideals $\tilde{\pfrak}_{\tilde{x}} \subseteq A_0[z]$ satisfying $\tilde{\pfrak}_{\tilde{x}} \cap A_0 = \pfrak_x$ and $z \notin \tilde{\pfrak}_{\tilde{x}}$; geometrically, these represent points in the ``total space away from the zero section''. By duality, the coaction $\tilde{\rho}$ induces an action of the underlying affine algebraic group $G_{0}$ on $\lvert \sSpec A[z] \rvert$. We denote the orbit of a lifting $\tilde{x}$ by $G_0 \cdot \tilde{x}$ and its closure by $\overline{G_{0}\cdot \tilde{x}}$. Within this framework, the ``zero section'' is given by the closed subset $V(z) \subseteq \lvert \sSpec A[z] \rvert$ defined as the vanishing locus of the variable $z$. We then have the following characterization.

\begin{lemma}\label{lem:4.8}
Let $x \in \lvert X \rvert$ and let $\tilde{x} \in \lvert \sSpec A[z] \rvert$ be a lifting of $x$. Then:
\begin{enumerate}
\item $x$ is $\chi$-semistable if and only if the orbit closure $\overline{G_0 \cdot \tilde{x}} \subseteq \lvert \sSpec A[z] \rvert$ is disjoint from $V(z)$. In particular, it is a necessary condition that $r_{\Delta}(\chi) = 1$ in the Hopf superalgebra $k[\Delta]$.

\item $x$ is $\chi$-stable if and only if its orbit $G_0 \cdot \tilde{x}$ is a closed subset of $\lvert \sSpec A[z] \rvert$ and, furthermore, the effective action of $G_0/\Delta_0$ on $\tilde{x}$ has a finite stabilizer.
\end{enumerate}
\end{lemma}

\begin{proof}
Observe first that if an element $a \in A$ satisfies $\rho(a) = a \otimes \chi^{n}$, then the element $az^{n} \in A[z]$ satisfies $\tilde{\rho}(az^{n}) = az^{n} \otimes 1$. This means that relative invariant elements in $A$ of weight $\chi^{n}$, give rise to invariant elements in $A[z]$.

We next establish the necessary condition stated above. Let $x \in \lvert X \rvert$ be a $\chi$-semistable point and take a relative invariant $f \in (A^{G,\chi^n})_0$ such that $f(x) \neq 0$. Since $f$ has weight $\chi^{n}$, applying the induced coaction $\rho_{\Delta} = (\id_A \otimes r_{\Delta}) \circ \rho$, which is trivial on $A$, yields the identity
$$
f \otimes 1 = \rho_{\Delta}(f) =  (\id_A \otimes r_{\Delta})(\rho(f)) = f \otimes r_{\Delta}(\chi)^n,
$$
which holds in the tensor product $A_0 \otimes_k k[\Delta]$. As $f$ is non-trivial, this equality implies that $r_{\Delta}(\chi)^n = 1$ in the Hopf superalgebra $k[\Delta]$. For this to hold independently of the choice of $n$, we need $r_{\Delta}(\chi) = 1$, as asserted. 

To prove the forward implication of (1), assume that $x \in \lvert X \rvert$ is $\chi$-semistable and let $f \in (A^{G,\chi^n})_0$ be as before, with $f(x) \neq 0$. By our first observation, $f z^n \in A_0[z]$ is an invariant element under $\tilde{\rho}$, which implies that the basic open subset $D(f z^n)$ of $\lvert \sSpec A[z] \rvert$ is $G_0$-invariant. Since $z \notin \tilde{\pfrak}_{\tilde{x}}$ by definition and $f \notin \pfrak_x$, it follows that $f z^n \notin \tilde{\pfrak}_{\tilde{x}}$, so $D(f z^n)$ contains $\tilde{x}$. In light of its $G_0$-invariance, $D(f z^n)$ must also contain the entire orbit $G_0 \cdot \tilde{x}$ and its closure $\overline{G_0 \cdot \tilde{x}}$. Finally, since $D(f z^n) \subseteq D(z)$, the orbit closure $\overline{G_{0}\cdot \tilde{x}}$ is disjoint from $V(z)$.

For the reverse implication of (1), assume that $\overline{G_0 \cdot \tilde{x}} \cap V(z) = \varnothing$. By the reductivity of $G_0$, there exists a $G_0$-invariant principal open subset $D(\tilde{f}) \subseteq D(z)$ containing $\tilde{x}$. We can choose its homogeneous generator $\tilde{f} \in A_0[z]$ to be an invariant element under $\tilde{\rho}$, so that $\tilde{\rho}(\tilde{f}) = \tilde{f} \otimes 1$. Writing $\tilde{f} = \sum_i f_i z^i$, the relation $\sum_i \rho(f_i) \otimes z^i \chi^{-i} = \sum_i f_i z^i \otimes 1$, together with the linear independence of the powers of $z$, implies that each coefficient $f_i \in A_0$ satisfies $\rho(f_i) = f_i \otimes \chi^i$. Since $\tilde{f} \notin \tilde{\pfrak}_{\tilde{x}}$, the polynomial structure implies that at least one coefficient satisfies $f_i(x) \neq 0$. Restricting the coaction to the underlying classical affine group $G_0$ yields $\rho_0(f_i) = f_i \otimes \chi_0^i$. The condition $r_{\Delta}(\chi) = 1$ ensures, via Lemma~\ref{lem:4.5}, that $\chi_0(\Delta_0) = \{ 1 \}$, so this weight assignment satisfies King's semistability criterion for $G_0$. Consequently, $f_i$ is an ordinary relative invariant of weight $\chi_0^i$, which proves the $\chi$-semistability of $x$.

To prove the forward implication of (2), assume that $x \in \lvert X \rvert$ is $\chi$-stable. By definition, there exists a relative invariant $f \in (A^{G,\chi^n})_0$ with $f(x) \neq 0$ such that the orbit $G_0 \cdot x$ is a closed subset of the basic open subset $D(f)$ that satisfies $\dim G_0 \cdot x = \dim G_0/\Delta_0$. Let us consider the canonical projection $\lvert \sSpec A[z] \rvert \to \lvert X \rvert$. This projection is $G_0$-equivariant and maps the orbit $G_0 \cdot \tilde{x}$ onto the orbit $G_0 \cdot x$. Since the stabilizer of $x$ is $\Delta_0$, and since $\chi_0$ is trivial on $\Delta_0$, as noted above, thus acting trivially on the variable $z$, any two elements in $G_0 \cdot \tilde{x}$ that project to the same point in the base $\lvert X \rvert$ must coincide. It follows that the canonical projection $\lvert \sSpec A[z] \rvert \to \lvert X \rvert$ restricts to an injective mapping between these orbits, implying that these orbits are in bijective correspondence. This bijectivity ensures that the lifted orbit $G_0 \cdot \tilde{x}$ is a closed subset of $\lvert \sSpec A[z] \rvert$ and preserves the relation $\dim G_0 \cdot \tilde{x} =  \dim G_0/\Delta_0$. At the same time, by the foregoing, the action of $G_0$ on $\lvert \sSpec A[z] \rvert$ factors through the quotient group $G_0/\Delta_0$. Applying the orbit dimension formula to this effective action yields the relation
$$
\dim G_0 \cdot \tilde{x} = \dim G_0/\Delta_0 - \dim (G_0/\Delta_0)_{\tilde{x}},
$$
where $(G_0/\Delta_0)_{\tilde{x}}$ denotes the stabilizer subgroup of $\tilde{x}$ under $G_0/\Delta_0$. Comparing both expressions for $\dim G_0 \cdot \tilde{x}$ shows that $\dim (G_0/\Delta_0)_{\tilde{x}} = 0$. Since this stabilizer is a closed subscheme of finite type, having dimension zero implies that it is finite.

For the reverse implication of (2), assume that the orbit $G_0 \cdot \tilde{x}$ is a closed subset of $\lvert \sSpec A[z]\rvert$ and that the stabilizer subgroup $(G_0/\Delta_0)_{\tilde{x}}$ is finite. The finiteness of the stabilizer implies that $\dim (G_0/\Delta_0)_{\tilde{x}} = 0$. Substituting this dimension into the orbit dimension formula yields $\dim G_0 \cdot \tilde{x} = \dim G_0/\Delta_0$. Since the canonical projection $\lvert \sSpec A[z]\rvert \to \lvert X\rvert$ preserves orbit dimensions as stated above, we obtain the dimension relation $\dim G_0 \cdot x = \dim G_0/\Delta_0$. Next, observe that since $z \notin \tilde{\pfrak}_{\tilde{x}}$, the point $\tilde{x}$ lies outside the $G_0$-invariant vanishing locus $V(z)$, which means that the entire orbit $G_0 \cdot \tilde{x}$ is disjoint from $V(z)$. Since this orbit is closed by hypothesis, its closure $\overline{G_0 \cdot \tilde{x}}$ remains disjoint from $V(z)$. By the reverse implication of (1), the point $x$ is $\chi$-semistable, which guarantees the existence of a relative invariant $f \in (A^{G,\chi^n})_0$ such that $f(x) \neq 0$. Considering the element $f z^n$ as in our first observation, the basic open subset $D(f z^n)$ is $G_0$-invariant and contains the entire orbit $G_0 \cdot \tilde{x}$. On this subset, the canonical projection restricts to a morphism $D(f z^n) \to D(f)$, under which the image of the closed orbit $G_0 \cdot \tilde{x}$ is the orbit $G_0 \cdot x$, itself a closed subset of the basic open subset $D(f)$. Consequently, the point $x$ is $\chi$-stable.
\end{proof}

In close analogy to King's construction, the GIT equivalence relation on $X_{\chi}^{\uss}$ can be characterized via an equivalence relation on $\lvert \sSpec A[z] \rvert$. Namely, $x \sim y$ if and only if there exist liftings $\tilde{x}$ and $\tilde{y}$ such that the orbit closures $\overline{G_0 \cdot \tilde{x}}$ and $\overline{G_0 \cdot \tilde{y}}$ intersect in $\lvert \sSpec A[z] \rvert$. This equivalence follows immediately from the bijective correspondence between these orbits established in the proof of Lemma~\ref{lem:4.8}.

Before proceeding further, we need to introduce some notation. If one is given a one-parameter subgroup $\lambda^{*} \colon k[G] \to k[\GG_m]$, one obtains the superalgebra morphisms
\begin{equation}
\begin{gathered} 
  \rho_{\lambda} = (\id_{A} \otimes \lambda^{*}) \circ \rho \colon A \to A \otimes_k k[\GG_m], \\     
   \tilde{\rho}_{\lambda} = (\id_{A[z]} \otimes \lambda^{*}) \circ \tilde{\rho} \colon A[z] \to A[z] \otimes_k k[\GG_m], 
  \end{gathered}
\end{equation}
which determine coactions of the multiplicative group $\GG_{m}$ on the coordinate rings $A$ and $A[z]$, respectively. Recalling that $k[\GG_m] = k[t,t^{-1}]$, for each $a \in k^\times$, the evaluation homomorphism $\ev_a \colon k[t,t^{-1}] \to k$ sending $t$ to $a$, further induces the invertible superalgebra morphisms
\begin{equation}
\begin{gathered}
    \sigma_{\lambda,a} = (\id_{A} \otimes \ev_a) \circ \rho_{\lambda} \colon A \to A, \\
    \tilde{\sigma}_{\lambda,a} = (\id_{A[z]} \otimes \ev_a) \circ \tilde{\rho}_{\lambda} \colon A[z] \to A[z]. 
\end{gathered}
\end{equation}
These superalgebra morphisms, in turn, yield continuous, invertible maps $\sigma_{\lambda,a}^*\colon \lvert X \rvert \to \lvert X \rvert$ and $ \tilde{\sigma}_{\lambda,a}^* \colon \lvert \sSpec A[z] \rvert \to \lvert \sSpec A[z] \rvert$, as discussed in \S\ref{sec:2.2}.\footnote{To lighten the notation, here we drop the bars when writing the induced maps on underlying topological spaces.} Therefore, given a point $x \in \lvert X \rvert$ and a lifting $\tilde{x} \in \lvert \sSpec A[z] \rvert$, one obtains the points $ \sigma_{\lambda,a}^*(x) \in \lvert X \rvert$ and $\tilde{\sigma}_{\lambda,a}^*(\tilde{x}) \in \lvert \sSpec A[z] \rvert$. The prime ideal corresponding to $\sigma_{\lambda,a}^*(x)$ is $\sigma_{\lambda,a}^{-1}(\mathfrak{p}_x)$, whereas the one corresponding to $\tilde{\sigma}_{\lambda,a}^*(\tilde{x})$ is $\tilde{\sigma}_{\lambda,a}^{-1}(\tilde{\mathfrak{p}}_{\tilde{x}})$. In this context, the following assertion holds.

\begin{lemma}\label{lem:4.9}
For each $a \in k^{\times}$, the points $\sigma_{\lambda,a}^*(x)$ and $\tilde{\sigma}_{\lambda,a}^*(\tilde{x})$ belong to the orbits $G_0 \cdot x$ and $G_0 \cdot \tilde{x}$, respectively.
\end{lemma}

\begin{proof}
It suffices to verify the assertion for $\sigma_{\lambda,a}^*(x)$, as the argument for the lifting $\tilde{\sigma}_{\lambda,a}^*(\tilde{x})$ proceeds identically. Substituting the expression for the induced coaction $\rho_{\lambda}$ into $\sigma_{\lambda,a}$ yields
$$
\sigma _{\lambda ,a} =  (\id_{A} \otimes (\ev_a \circ \lambda^{*})) \circ \rho.
$$
Thus, $\sigma_{\lambda,a}$ is obtained by composing the original coaction $\rho$ with the base extension of the superalgebra homomorphism $\xi_{\lambda,a} = \ev_a \circ \lambda^{*} \colon k[G] \to k$. As pointed out in \S\ref{sec:3.1}, this map determines a $k$-point of the algebraic supergroup $G$. Since $\xi_{\lambda,a}$ factors through the canonical restriction $r_0 \colon k[G] \to k[G_0]$ onto the coordinate ring of the underlying affine algebraic group $G_0$, it defines an algebra map $\bar{\xi}_{\lambda,a} \colon k[G_0] \to k$, corresponding to a $k$-point of $G_0$. This factorization yields the identity $\sigma_{\lambda,a} = (\id_A \otimes \bar{\xi}_{\lambda,a}) \circ \rho_0$, where $\rho_0$ denotes the standard coaction of $G_0$. Consequently, the endomorphism $\sigma_{\lambda,a}$ coincides with the action of $G_0$ on $A$ associated with this $k$-point, which by definition implies that $\sigma_{\lambda,a}^*(x) \in G_0 \cdot x$.
\end{proof}

One may interpret this result as saying that, as $a$ ranges over $k^{\times}$, the points $\sigma_{\lambda,a}^*(x)$ and $\tilde{\sigma}_{\lambda,a}^*(\tilde{x})$ trace out the movement along the respective orbits under the one-parameter subgroup. In this sense, they give precise meaning to the otherwise ill-defined expressions ``$\lambda(a) \cdot x$'' and ``$\lambda(a) \cdot \tilde{x}$''.  

We now come to the key definition. By a \emph{polynomial coextension} of the coaction $\rho_{\lambda} \colon A \to A \otimes_k k[\mathbb{G}_m]$ induced by $\lambda^{*} \colon k[G] \to k[\mathbb{G}_m]$, we mean a morphism of superalgebras $\hat{\rho}_{\lambda} \colon A \to A \otimes_k k[\mathbb{A}^{1}]$ making the following diagram commute
\begin{equation}
\begin{tikzcd}[row sep=3.5em,  
column sep=4.0em,
  every label/.append style={font=\normalsize}
]
  & A \otimes_k  k[\AA^{1}]  \arrow[d, -{To[length=2.5pt, width=4pt]}] \\
A \arrow[r, -{To[length=2.5pt, width=4pt]}, "\rho_{\lambda}"'] \arrow[ru, -{To[length=2.5pt, width=4pt]}, "\hat{\rho}_{\lambda}"]& A \otimes_k  k[\GG_m] 
\end{tikzcd}
\end{equation}
where the vertical arrow on the right is the morphism of superalgebras induced by the natural inclusion $k[\mathbb{A}^{1}] \hookrightarrow k[\mathbb{G}_m]$. Similarly, by a polynomial coextension of the lifted coaction $\tilde{\rho}_{\lambda} \colon A[z] \to A[z] \otimes_k k[\mathbb{G}_m]$, we mean a morphism of superalgebras $\hat{\tilde{\rho}}_{\lambda} \colon A[z] \to A[z] \otimes_k k[\mathbb{A}^{1}]$ making an analogous diagram commute. Crucially, the existence of a polynomial coextension for $\rho_{\lambda}$ does not automatically ensure the existence of one for $\tilde{\rho}_{\lambda}$; the precise condition for this extension to hold will be established shortly. Nevertheless, to set up the necessary notation, let us temporarily suppose that both polynomial coextensions exist. This way, writing $k[\mathbb{A}^1] = k[t]$, for each $a \in k$ the evaluation homomorphism $\ev_a \colon k[t] \to k$ setting $t$ to $a$ induces the superalgebra morphisms
\begin{equation}\label{eq:4.14}
\begin{gathered}
    \hat{\sigma}_{\lambda,a} = (\id_{A} \otimes \ev_a) \circ \hat{\rho}_{\lambda} \colon A \to A, \\
    \hat{\tilde{\sigma}}_{\lambda,a} = (\id_{A[z]} \otimes \ev_a) \circ \hat{\tilde{\rho}}_{\lambda} \colon A[z] \to A[z]. 
\end{gathered}
\end{equation}
From these superalgebra endomorphisms, one then obtains the continuous maps $\hat{\sigma}_{\lambda,a}^*\colon \lvert X \rvert \to \lvert X \rvert$ and $\hat{\tilde{\sigma}}_{\lambda,a}^* \colon \lvert \sSpec A[z] \rvert \to \lvert \sSpec A[z] \rvert$. With all this understanding, given a point $x \in \lvert X \rvert$ and a lifting $\tilde{x} \in \lvert \sSpec A[z] \rvert$, one defines the algebraic limits
\begin{equation}
\begin{aligned}
\lim _{a\rightarrow 0}\sigma _{\lambda ,a}^{*}(x)&=\hat{\sigma} _{\lambda,0}^{*}(x),\\
 \lim _{a\rightarrow 0}\tilde{\sigma }_{\lambda ,a}^{*}(\tilde{x})&= \hat{\tilde{\sigma }}_{\lambda ,0}^{*}(\tilde{x}).
\end{aligned}
\end{equation}
These definitions serve to give precise mathematical meaning to what are a priori purely formal expressions, ``$\displaystyle \lim_{a \to 0} \lambda(a) \cdot x$'' and ``$\displaystyle \lim_{a \to 0} \lambda(a) \cdot \tilde{x}$''. In this regard, the following proposition is key to our development.

\begin{lemma}\label{lem:4.10}
Let $\tilde{x} \in \lvert \sSpec A[z] \rvert$ be a lift of $x \in \lvert X \rvert$, as above. Then:
\begin{enumerate}
\item $x$ is $\chi$-semistable if and only if, for all one-parameter subgroups $\lambda^{*} \colon k[G] \to k[\GG_m]$ for which the induced lifted coaction $\tilde{\rho}_{\lambda}$ admits a polynomial coextension, one has $\displaystyle \lim _{a\rightarrow 0}\tilde{\sigma }_{\lambda ,a}^{*}(\tilde{x})\notin V(z)$.

\item $x$ is $\chi$-stable if and only if, for all one-parameter subgroups $\lambda^{*} \colon k[G] \to k[\GG_m]$ for which the induced lifted coaction $\tilde{\rho}_{\lambda}$ admits a polynomial coextension, one has $\displaystyle \lim_{a \to 0} \tilde{\sigma}_{\lambda,a}^*(\tilde{x}) \in G_0 \cdot \tilde{x}$, where the equality $\displaystyle \lim_{a \to 0} \tilde{\sigma}_{\lambda,a}^*(\tilde{x}) = \tilde{x}$ occurs only when $\lambda ^{*}$ factors through $k[G_0/\Delta_0]$.
\end{enumerate}
\end{lemma}

\begin{proof}
Suppose that $x$ is $\chi$-semistable. By Lemma~\ref{lem:4.8}, this condition is equivalent to the assertion that the orbit closure $\overline{G_0 \cdot \tilde{x}}$ does not intersect the vanishing locus $V(z)$. Let $\lambda^* \colon k[G] \to k[\GG_m]$ be a one-parameter subgroup for which the induced lifted coaction $\tilde{\rho}_{\lambda}$ admits a polynomial coextension. Since the parameterized points $\tilde{\sigma}_{\lambda,a}^*(\tilde{x})$ belong to the orbit $G_0 \cdot \tilde{x}$ for each $a \in k^\times$ by Lemma~\ref{lem:4.9}, any regular function vanishing on the orbit also vanishes at their algebraic limit $\displaystyle \lim_{a \to 0} \tilde{\sigma}_{\lambda,a}^*(\tilde{x})$. Consequently, this limit belongs to the orbit closure $\overline{G_0 \cdot \tilde{x}}$. Since this closure is disjoint from $V(z)$, it follows that $\displaystyle\lim_{a \to 0} \tilde{\sigma}_{\lambda,a}^*(\tilde{x}) \notin V(z)$.

Conversely, assume that for all one-parameter subgroups $\lambda^* \colon k[G] \to k[\GG_m]$ for which the induced lifted coaction $\tilde{\rho}_{\lambda}$ admits a polynomial coextension, the limit $\displaystyle \lim_{a \to 0} \tilde{\sigma}_{\lambda,a}^*(\tilde{x})$ does not belong to $V(z)$. Assume, to the contrary, that the orbit closure $\overline{G_0 \cdot \tilde{x}}$ intersects the vanishing locus $V(z)$, and choose a point within this intersection. By the classical Hilbert--Mumford criterion applied to the action of $G_0$, there exists an ordinary one-parameter subgroup $\lambda_0^* \colon k[G_0] \to k[\GG_m]$ whose associated limit is this intersection point. As detailed in \S\ref{sec:2.4}, this ordinary map determines a unique one-parameter subgroup $\lambda^* \colon k[G] \to k[\mathbb{G}_m]$, which naturally carries a polynomial coextension by virtue of its definition from $\lambda_0^*$. Because the algebraic limit $\displaystyle \lim_{a \to 0} \tilde{\sigma}_{\lambda,a}^*(\tilde{x})$ coincides with this intersection point, it belongs to $V(z)$, contradicting the hypothesis.

For the second assertion, suppose that $x$ is $\chi$-stable. By Lemma~\ref{lem:4.8}, this implies that the classical orbit $G_0 \cdot \tilde{x}$ is closed and does not intersect the vanishing locus $V(z)$. Let $\lambda^* \colon k[G] \to k[\GG_m]$ be a one-parameter subgroup for which the induced lifted coaction $\tilde{\rho}_{\lambda}$ admits a polynomial coextension. Following the reasoning of the first assertion, the algebraic limit $\displaystyle \lim_{a \to 0} \tilde{\sigma}_{\lambda,a}^*(\tilde{x})$ belongs to the orbit closure $\overline{G_0 \cdot \tilde{x}}$, which coincides with the orbit itself since the latter is closed by hypothesis. One therefore has $\displaystyle \lim_{a \to 0} \tilde{\sigma}_{\lambda,a}^*(\tilde{x}) \in G_0 \cdot \tilde{x}$. Assume furthermore that this algebraic limit satisfies the equality $\displaystyle \lim_{a \to 0} \tilde{\sigma}_{\lambda,a}^*(\tilde{x}) = \tilde{x}$. Appealing again to \S\ref{sec:2.4}, the morphism $\lambda^*$ is uniquely determined by an ordinary one-parameter subgroup $\lambda_0^* \colon k[G_0] \to k[\GG_m]$. The assumption that the algebraic limit fixes the point $\tilde{x}$ implies that the image of this ordinary subgroup is contained within the stabilizer of $\tilde{x}$ under the action of $G_0$. Because the action of the quotient group $G_0/\Delta_0$ has a finite stabilizer at $\tilde{x}$, any ordinary one-parameter subgroup valued in this stabilizer must belong to the kernel of the canonical projection $G_0 \to G_0/\Delta_0$. Consequently, the ordinary map $\lambda_0^*$ factors through $k[G_0/\Delta_0]$, which means that the original one-parameter subgroup $\lambda^*$ factors through $k[G_0/\Delta_0]$ as well.

Conversely, assume that for all one-parameter subgroups $\lambda^* \colon k[G] \to k[\GG_m]$ for which the induced lifted coaction $\tilde{\rho}_{\lambda}$ admits a polynomial coextension, one has $\displaystyle \lim_{a \to 0} \tilde{\sigma}_{\lambda,a}^*(\tilde{x}) \in G_0 \cdot \tilde{x}$, where the equality $\displaystyle \lim_{a \to 0} \tilde{\sigma}_{\lambda,a}^*(\tilde{x}) = \tilde{x}$ occurs only when $\lambda^*$ factors through $k[G_0/\Delta_0]$. Let $\tilde{y}$ be a point in the orbit closure $\overline{G_0 \cdot \tilde{x}}$. By the classical Hilbert--Mumford criterion applied to the action of $G_0$, there exists an ordinary one-parameter subgroup $\lambda_0^* \colon k[G_0] \to k[\GG_m]$ whose associated limit is precisely $\tilde{y}$. As above, this ordinary map determines a unique one-parameter subgroup $\lambda^* \colon k[G] \to k[\mathbb{G}_m]$ of $G$, which likewise admits a polynomial coextension. By construction, its algebraic limit $\displaystyle \lim_{a \to 0} \tilde{\sigma}_{\lambda,a}^*(\tilde{x})$ coincides with $\tilde{y}$. Since this limit avoids the vanishing locus $V(z)$ by the first assertion, the hypothesis forces the containment $\displaystyle \lim_{a \to 0} \tilde{\sigma}_{\lambda,a}^*(\tilde{x}) \in G_0 \cdot \tilde{x}$. It follows that $\tilde{y} \in G_0 \cdot \tilde{x}$, which proves that the classical orbit $G_0 \cdot \tilde{x}$ is a closed subset. In light of Lemma~\ref{lem:4.8}, the $\chi$-stability of $x$ follows as soon as we establish that the effective action of $G_0/\Delta_0$ at $\tilde{x}$ has a finite stabilizer. Suppose the contrary. The effective stabilizer must then contain a non-trivial algebraic subgroup of positive dimension, which in turn guarantees the existence of an ordinary one-parameter subgroup $\lambda_0^* \colon k[G_0] \to k[\GG_m]$ whose composition with the quotient map remains non-trivial and fixes the point $\tilde{x}$. As before, this ordinary map induces a unique one-parameter subgroup $\lambda^* \colon k[G] \to k[\GG_m]$. Because the ordinary subgroup fixes $\tilde{x}$, the algebraic limit under the associated polynomial coextension satisfies the identity $\displaystyle \lim_{a \to 0} \tilde{\sigma}_{\lambda,a}^*(\tilde{x}) = \tilde{x}$. By hypothesis, this equality implies that the original one-parameter subgroup $\lambda^*$ factors through $k[G_0/\Delta_0]$. Equivalently, the ordinary map $\lambda_0^*$ must factor through $k[G_0/\Delta_0]$, contradicting the non-triviality of its image in the quotient group. This contradiction establishes that the effective action of $G_0/\Delta_0$ has a finite stabilizer at $\tilde{x}$. 
\end{proof}

The geometric characterization established in the previous lemma describes the semistability and stability of a point $x$ through its lifting $\tilde{x}$ and the behavior of the lifted coaction $\tilde{\rho}_{\lambda}$ on the superalgebra $A[z]$. To reformulate these conditions strictly in terms of the coaction $\rho_{\lambda}$ on $A$, without invoking the lifting $\tilde{x}$, one must introduce a pairing that isolates the action on the even variable $z$. For a one-parameter subgroup $\lambda^* \colon k[G] \to k[\mathbb{G}_m]$ and a group-like element $\chi \in k[G]_0^\times$, the image of $\chi$ under $\lambda^*$ is necessarily an invertible element of $k[\mathbb{G}_m] = k[t,t^{-1}]$ preserving the counit structure. Consequently, there exists a unique expression of the form
\begin{equation}
\lambda^*(\chi) = t^m
\end{equation}
where $m \in \mathbb{Z}$. We shall denote this integer by $\langle \chi, \lambda \rangle$. This assignment can be related directly to the behavior on the underlying affine algebraic group $G_0$. Indeed, since the one-parameter subgroup satisfies the identity $\lambda^* = \lambda_0^* \circ r_0$ for a unique ordinary one-parameter subgroup $\lambda_0^* \colon k[G_0] \to k[\mathbb{G}_m]$, evaluating this superalgebra morphism on the group-like element yields the chain of equalities
\begin{equation}
\lambda^*(\chi) = (\lambda_0^* \circ r_0)(\chi) = \lambda_0^*(r_0(\chi)) = \lambda_0^*(\chi_0),
\end{equation}
from which it follows that the invariant $\langle \chi, \lambda \rangle$ can be evaluated equivalently as the pairing $\langle \chi_0, \lambda_0 \rangle$ originally considered by King. The following is a crucial observation. 

\begin{lemma}\label{lem:4.11}
For every fixed one-parameter subgroup $\lambda^* \colon k[G] \to k[\GG_m]$, the induced lifted coaction $\tilde{\rho}_{\lambda} \colon A[z] \to A[z] \otimes_k k[\GG_m]$ admits a polynomial coextension if and only if the induced coaction $\rho_{\lambda} \colon A \to A \otimes_k k[\GG_m]$ admits a polynomial coextension and $\langle \chi, \lambda \rangle \le 0$ holds for every group-like element $\chi \in k[G]_0^\times$.
\end{lemma}

\begin{proof}
Suppose first that the lifted coaction $\tilde{\rho}_\lambda$ admits a polynomial coextension, which is to say that we have an extended superalgebra morphism $\hat{\tilde{\rho}}_\lambda \colon A[z] \to A[z] \otimes_k k[\AA^1]$. By restricting this morphism to the subsuperalgebra $A \hookrightarrow A[z]$, and noting that the lifted coaction restricts to the original coaction on $A$, we naturally obtain a polynomial coextension $\hat{\rho}_\lambda \colon A \to A \otimes_k k[\AA^1]$ for $\rho_\lambda$. On the other hand, the extension $\hat{\tilde{\rho}}_\lambda$ must agree with $\tilde{\rho}_\lambda$ on $A[z]$ via the canonical inclusion $k[\mathbb{A}^1] \hookrightarrow k[\GG_m]$. Evaluating the coaction on the even generator $z$ yields
$$
\tilde{\rho}_\lambda(z) = z \otimes \lambda^*(\chi^{-1}) = z \otimes t^{-\langle \chi, \lambda \rangle}.
$$
For this element to lie in the polynomial subsuperalgebra $A[z] \otimes_k k[\AA^1] = A[z] \otimes_k k[t]$, the exponent of the variable $t$ cannot be negative. Consequently, for every group-like element $\chi \in k[G]_0^\times$, we must have $-\langle \chi, \lambda \rangle \ge 0$, which is equivalent to $\langle \chi, \lambda \rangle \le 0$.

Conversely, assume that the coaction $\rho_\lambda$ admits a polynomial coextension $\hat{\rho}_\lambda \colon A \to A \otimes_k k[\AA^1]$ and that the inequality $\langle \chi, \lambda \rangle \le 0$ holds for every group-like element $\chi \in k[G]_0^\times$. To construct a polynomial coextension for the lifted coaction, we apply the universal property of the free even polynomial extension $A[z]$. We define a superalgebra morphism $\hat{\tilde{\rho}}_\lambda \colon A[z] \to A[z] \otimes_k k[\AA^1]$ by specifying its values on the generators: it restricts to $\hat{\rho}_\lambda$ on $A$, and it maps the even variable $z$ to $z \otimes t^{-\langle \chi, \lambda \rangle}$, where $\chi \in k[G]_0^\times$ denotes the weight of $z$. Since $\langle \chi, \lambda \rangle \le 0$ by hypothesis, the exponent $-\langle \chi, \lambda \rangle$ is a non-negative integer, which guarantees that $t^{-\langle \chi, \lambda \rangle}$ belongs to $k[\AA^1] = k[t]$. By construction, $\hat{\tilde{\rho}}_\lambda$ makes the required diagram commute, thereby yielding the unique polynomial coextension for $\tilde{\rho}_\lambda$ and completing the proof.
\end{proof}

Through the structural constraints established in Lemma~\ref{lem:4.11}, the invariant $\langle \chi, \lambda \rangle$ determines how the endomorphisms $\tilde{\sigma}_{\lambda,a}$ scale the direction introduced by $z$ in $A[z]$, thereby governing whether the lifted coaction admits a polynomial coextension. With this interpretation, the classical Hilbert--Mumford numerical criterion translates into the following statement, formulated purely in terms of coactions and polynomial coextensions.

\begin{theorem}
Let $x \in \lvert X \rvert$ and let $\chi^* \colon k[\GG_m] \to k[G]$ be a character with associated group-like element $\chi \in k[G]_0^\times$. Then:
\begin{enumerate}
\item $x$ is $\chi$-semistable if and only if $r_\Delta(\chi) = 1$ holds in the Hopf superalgebra $k[\Delta]$, and every one-parameter subgroup $\lambda^* \colon k[G] \to k[\GG_m]$ for which the induced coaction $\rho_\lambda$ admits a polynomial coextension satisfies $\langle \chi, \lambda \rangle \geq 0$.

\item $x$ is $\chi$-stable if and only if it is $\chi$-semistable and the only one-parameter subgroups $\lambda^* \colon k[G] \to k[\GG_m]$ for which the induced coaction $\rho_\lambda$ admits a polynomial coextension and $\langle \chi, \lambda \rangle = 0$ holds are those that factor through $k[G_0/\Delta_0]$.
\end{enumerate}
\end{theorem}

\begin{proof}
We proceed by translating the geometric conditions of Lemma~\ref{lem:4.10} into numerical conditions on the invariant $\langle \chi, \lambda \rangle$ using the dictionary provided by Lemma~\ref{lem:4.11}.

Suppose first that $x \in \lvert X \rvert$ is $\chi$-semistable. By Lemma~\ref{lem:4.8}, it is an immediate necessary condition that the structural relation $r_\Delta(\chi) = 1$ holds in the Hopf superalgebra $k[\Delta]$. Now, consider an arbitrary one-parameter subgroup $\lambda^* \colon k[G] \to k[\GG_m]$ for which the induced coaction $\rho_\lambda$ admits a polynomial coextension. We wish to show that $\langle \chi, \lambda \rangle \ge 0$. Assume, to the contrary, that $\langle \chi, \lambda \rangle < 0$. Since this satisfies the weaker inequality $\langle \chi, \lambda \rangle \le 0$, Lemma~\ref{lem:4.11} guarantees that the induced lifted coaction $\tilde{\rho}_\lambda$ also admits a polynomial coextension, ensuring that the algebraic limit $\displaystyle\lim_{a \to 0} \tilde{\sigma}_{\lambda,a}^*(\tilde{x})$ is well-defined in $\lvert \sSpec A[z] \rvert$. Since $x$ is $\chi$-semistable, Lemma~\ref{lem:4.10} dictates that this algebraic limit must lie outside the vanishing locus $V(z)$. However, evaluating the even generator $z$ under the extended superalgebra morphism $\hat{\tilde{\sigma}}_{\lambda,0}$, which is calculated via the second relation in \eqref{eq:4.14}, yields
$$
\hat{\tilde{\sigma}}_{\lambda,0}(z) = (\id_{A[z]} \otimes \ev_0)(z \otimes t^{-\langle \chi ,\lambda \rangle }) = \ev_0(t^{-\langle \chi ,\lambda \rangle }) z .
$$
By our contradiction hypothesis, the exponent satisfies $-\langle \chi, \lambda \rangle > 0$. Evaluating this positive power of $t$ at $0$ causes the expression to collapse to zero, forcing $\hat{\tilde{\sigma}}_{\lambda,0}(z) = 0$. Geometrically, this means that the algebraic limit $\displaystyle\lim_{a \to 0} \tilde{\sigma}_{\lambda,a}^*(\tilde{x})$ belongs to the vanishing locus $V(z)$, directly contradicting Lemma~\ref{lem:4.10}. Thus, we must have $\langle \chi, \lambda \rangle \ge 0$.

Conversely, assume that $r_\Delta(\chi) = 1$ and that $\langle \chi, \lambda \rangle \ge 0$ for every one-parameter subgroup for which $\rho_\lambda$ admits a polynomial coextension. To prove that $x$ is $\chi$-semistable, it suffices by Lemma~\ref{lem:4.10} to show that for any one-parameter subgroup for which the lifted coaction $\tilde{\rho}_\lambda$ admits a polynomial coextension, the algebraic limit $\displaystyle\lim_{a \to 0} \tilde{\sigma}_{\lambda,a}^*(\tilde{x})$ does not lie in $V(z)$. If $\tilde{\rho}_\lambda$ admits a polynomial coextension, Lemma~\ref{lem:4.11} implies that $\rho_\lambda$ admits a polynomial coextension and $\langle \chi, \lambda \rangle \le 0$ simultaneously. Combining this with our global assumption $\langle \chi, \lambda \rangle \ge 0$ forces $\langle \chi, \lambda \rangle = 0$. Under the condition $\langle \chi, \lambda \rangle = 0$, the morphism $\hat{\tilde{\sigma}}_{\lambda,0}$ maps $z$ to $\ev_0(t^0) z = z  \neq 0$. Since the generator $z$ does not vanish under this map, the algebraic limit $\displaystyle\lim_{a \to 0} \tilde{\sigma}_{\lambda,a}^*(\tilde{x})$ is disjoint from $V(z)$. Therefore, by Lemma~\ref{lem:4.10}, the point $x$ is $\chi$-semistable.

Regarding the second assertion, assume first that $x \in \lvert X \rvert$ is $\chi$-stable. By the forward implication of Lemma~\ref{lem:4.10}, for every one-parameter subgroup $\lambda^*$ for which the induced lifted coaction $\tilde{\rho}_\lambda$ admits a polynomial coextension, the algebraic limit $\displaystyle\lim_{a \to 0} \tilde{\sigma}_{\lambda,a}^*(\tilde{x})$ belongs to the classical orbit $G_0 \cdot \tilde{x}$, and the equality $\lim_{a \to 0} \tilde{\sigma}_{\lambda,a}^*(\tilde{x}) = \tilde{x}$ occurs only when $\lambda^*$ factors through $k[G_0/\Delta_0]$. Now, let $\lambda^*$ be a non-trivial one-parameter subgroup for which the induced coaction $\rho_\lambda$ admits a polynomial coextension and satisfies $\langle \chi, \lambda \rangle = 0$. By Lemma~\ref{lem:4.11}, since $\langle \chi, \lambda \rangle = 0 \le 0$, the induced lifted coaction $\tilde{\rho}_\lambda$ also admits a polynomial coextension. Furthermore, since $\langle \chi, \lambda \rangle = 0$, the generator $z$ remains invariant under $\hat{\tilde{\sigma}}_{\lambda,0}$, which yields $\displaystyle\lim_{a \to 0} \tilde{\sigma}_{\lambda,a}^*(\tilde{x}) = \tilde{x}$. Applying the characterization of $\chi$-stability from Lemma~\ref{lem:4.10} to this identity, we conclude that $\lambda^*$ must factor through $k[G_0/\Delta_0]$.

Conversely, assume that $x$ is $\chi$-semistable and that the only one-parameter subgroups for which $\rho_\lambda$ admits a polynomial coextension and $\langle \chi, \lambda \rangle = 0$ holds are those that factor through $k[G_0/\Delta_0]$. To show that $x$ is $\chi$-stable via Lemma~\ref{lem:4.10}, let $\lambda^*$ be any one-parameter subgroup for which the induced lifted coaction $\tilde{\rho}_\lambda$ admits a polynomial coextension. By Lemma~\ref{lem:4.11}, this means that $\rho_\lambda$ admits a polynomial coextension and $\langle \chi, \lambda \rangle \le 0$ simultaneously. Since $x$ is $\chi$-semistable, the proof of the first assertion already guarantees that $\langle \chi, \lambda \rangle \ge 0$ for any such subgroup, forcing $\langle \chi, \lambda \rangle = 0$. As before, the condition $\langle \chi, \lambda \rangle = 0$ ensures that $z$ is preserved, yielding $\displaystyle\lim_{a \to 0} \tilde{\sigma}_{\lambda,a}^*(\tilde{x}) = \tilde{x}$. Since $\langle \chi, \lambda \rangle = 0$, our hypothesis applies directly, ensuring that $\lambda^*$ factors through $k[G_0/\Delta_0]$. This matches the exact condition required by the converse direction of Lemma~\ref{lem:4.10}, establishing that $x$ is $\chi$-stable and completing the proof.
\end{proof}

To close, following King's approach, it will be useful to characterize the geometric closedness of orbits and the GIT equivalence relation among $\chi$-semistable points purely in terms of coactions and their polynomial coextensions.

\begin{proposition}
Let the context be as above. Then:
\begin{enumerate}
\item An orbit $G_0 \cdot x$ is closed in $X_{\chi}^{\mathrm{ss}}$ if and only if, for every one-parameter subgroup $\lambda^* \colon k[G] \to k[\GG_m]$ for which the induced coaction $\rho_{\lambda}$ admits a polynomial coextension and $\langle \chi,\lambda \rangle = 0$, one has $\displaystyle\lim_{a \to 0} \sigma_{\lambda,a}^* (x) \in G_0 \cdot x$.

\item If $x,y \in X_{\chi}^{\mathrm{ss}}$, then $x \sim y$ if and only if there are one-parameter subgroups $\lambda_1,\lambda_2 \colon k[G] \to k[\mathbb{G}_m]$ such that $\langle \chi,\lambda_1 \rangle = \langle \chi,\lambda_2 \rangle = 0$ and $\displaystyle\lim_{a \to 0} \sigma_{\lambda_1,a}^* (x)$ and $\displaystyle\lim_{a \to 0} \sigma_{\lambda_2,a}^* (y)$ lie in the same closed $G_0$-orbit.
\end{enumerate}
\end{proposition}

\begin{proof}
We establish both characterizations by invoking the geometric criteria on the lifted space from Lemma~\ref{lem:4.8} and translating them to the base space via the algebraic dictionary of Lemma~\ref{lem:4.11}.

To establish the first assertion, let $\tilde{x} \in \lvert \sSpec A[z] \rvert$ be a lift of $x \in X_{\chi}^{\uss}$. Suppose first that the orbit $G_0 \cdot x$ is closed in $X_{\chi }^{\uss}$. By part (2) of Lemma~\ref{lem:4.8}, this is equivalent to the lifted orbit $G_0 \cdot \tilde{x}$ being closed in the open complement of $V(z)$. Let $\lambda^* \colon k[G] \to k[\GG_m]$ be a one-parameter subgroup for which the coaction $\rho _{\lambda }$ admits a polynomial coextension and $\langle \chi,\lambda \rangle = 0$. By Lemma~\ref{lem:4.11}, since $\langle \chi, \lambda \rangle = 0 \le 0$, the induced lifted coaction $\tilde{\rho }_{\lambda }$ also admits a polynomial coextension. Thus, the forward direction of part (2) of Lemma~\ref{lem:4.10} ensures that the algebraic limit $\displaystyle\lim_{a \to 0} \tilde{\sigma}_{\lambda,a}^*(\tilde{x})$ belongs to $G_0 \cdot \tilde{x}$. Because the equality $\langle \chi, \lambda \rangle = 0$ entails that $\hat{\tilde{\sigma}}_{\lambda,0}$ fixes $z$, as noted above, this algebraic limit remains disjoint from $V(z)$ and maps under the canonical projection $\lvert \sSpec A[z] \rvert \to \lvert X \rvert$ to the algebraic limit $\displaystyle\lim_{a \to 0} \sigma_{\lambda,a}^*(x)$. Since the canonical projection is $G_0$-equivariant and maps $G_0 \cdot \tilde{x}$ onto $G_0 \cdot x$, this relation directly yields $\displaystyle\lim_{a \to 0} \sigma_{\lambda,a}^*(x) \in G_0 \cdot x$.

Conversely, assume that for every one-parameter subgroup satisfying $\langle \chi,\lambda \rangle = 0$ where $\rho _{\lambda }$ admits a polynomial coextension, one has $\displaystyle\lim_{a \to 0} \sigma_{\lambda,a}^*(x) \in G_0 \cdot x$. Let $\lambda ^{*}$ be any one-parameter subgroup for which the induced lifted coaction $\tilde{\rho }_{\lambda }$ admits a polynomial coextension. By Lemma~\ref{lem:4.11}, this implies that $\rho _{\lambda }$ admits a polynomial coextension and $\langle \chi, \lambda \rangle \le 0$. Since $x$ is $\chi$-semistable, the numerical criterion guarantees that $\langle \chi, \lambda \rangle \ge 0$, forcing $\langle \chi, \lambda \rangle = 0$. Under this condition, our hypothesis applies, so $\displaystyle\lim_{a \to 0} \sigma_{\lambda,a}^*(x) \in G_0 \cdot x$. Once again, the condition $\langle \chi, \lambda \rangle = 0$ ensures that $\hat{\tilde{\sigma}}_{\lambda,0}$ fixes $z$. Consequently, the lifted algebraic limit $\lim_{a \to 0} \tilde{\sigma}_{\lambda,a}^*(\tilde{x})$ is well-defined, remains disjoint from $V(z)$, and projects onto $\displaystyle\lim_{a \to 0} \sigma_{\lambda,a}^*(x)$, ensuring that $\displaystyle\lim_{a \to 0} \tilde{\sigma}_{\lambda,a}^*(\tilde{x}) \in G_0 \cdot \tilde{x}$. By the converse of part (2) of Lemma~\ref{lem:4.10}, the lifted orbit $G_0 \cdot \tilde{x}$ is closed in the complement of $V(z)$, which by part (2) of Lemma~\ref{lem:4.8} establishes that $G_0 \cdot x$ is closed in $X_{\chi }^{\uss}$.

To establish the second assertion, assume first that $x \sim y$ for $x,y \in X_{\chi}^{\mathrm{ss}}$. By the characterization of GIT equivalence in Lemma~\ref{lem:4.8}, this is equivalent to the condition that the lifted orbit closures $\overline{G_{0}\cdot \tilde{x}}$ and $\overline{G_{0}\cdot \tilde{y}}$ intersect outside $V(z)$. By the classical theory of reductive group actions, this intersection contains a unique closed $G_{0}$-orbit outside $V(z)$. Let $\tilde{u}$ be a point in this closed orbit, which necessarily lies outside $V(z)$. Applying the forward direction of Lemma~\ref{lem:4.10} to each lifting independently, there exist one-parameter subgroups $\lambda _{1}^{*}$ and $\lambda _{2}^{*}$ for which the induced lifted coactions admit polynomial coextensions such that $\displaystyle\lim_{a \to 0} \tilde{\sigma}_{\lambda_1,a}^*(\tilde{x}) = \tilde{u}$ and $\displaystyle\lim_{a \to 0} \tilde{\sigma}_{\lambda_2,a}^*(\tilde{y}) = \tilde{u}$. Because $\tilde{u} \notin V(z)$, Lemma~\ref{lem:4.11} implies that the coactions $\rho _{\lambda _{1}}$ and $\rho _{\lambda _{2}}$ admit polynomial coextensions and $\langle \chi,\lambda_1 \rangle = \langle \chi,\lambda_2 \rangle = 0$. Projecting these mappings onto the base space $\lvert X \rvert$ shows that the algebraic limits $\displaystyle\lim_{a \to 0} \sigma_{\lambda_1,a}^*(x)$ and $\displaystyle\lim_{a \to 0} \sigma_{\lambda_2,a}^*(y)$ coincide with the image of $\tilde{u}$ under the canonical projection, landing in the same closed base orbit.

Conversely, assume that there exist one-parameter subgroups $\lambda_1^*, \lambda_2^*$ for which the induced coactions admit polynomial coextensions, $\langle \chi,\lambda_1 \rangle = \langle \chi,\lambda_2 \rangle = 0$, and their algebraic limits $\displaystyle\lim_{a \to 0} \sigma_{\lambda_1,a}^*(x)$ and $\displaystyle\lim_{a \to 0} \sigma_{\lambda_2,a}^*(y)$ lie in the same closed $G_{0}$-orbit. By Lemma~\ref{lem:4.11}, since $\langle \chi, \lambda_i \rangle = 0 \le 0$, the induced lifted coactions $\tilde{\rho }_{\lambda _{1}}$ and $\tilde{\rho }_{\lambda _{2}}$ also admit polynomial coextensions. Because $\langle \chi, \lambda_i \rangle = 0$ implies that each $\hat{\tilde{\sigma}}_{\lambda_i,0}$ fixes $z$, the algebraic limits $\displaystyle\lim_{a \to 0} \tilde{\sigma}_{\lambda_1,a}^*(\tilde{x})$ and $\displaystyle\lim_{a \to 0} \tilde{\sigma}_{\lambda_2,a}^*(\tilde{y})$ remain disjoint from $V(z)$ and land in a common closed lifted orbit. By the forward direction of Lemma~\ref{lem:4.10}, these algebraic limits belong to the respective lifted orbit closures $\overline{G_{0}\cdot \tilde{x}}$ and $\overline{G_{0}\cdot \tilde{y}}$. Since their limits land in the same orbit, the closures intersect outside $V(z)$, which by the discussion following Lemma~\ref{lem:4.8} implies that the base points satisfy $x \sim y$, completing the proof.
\end{proof}

We wrap up this discussion with a few examples.

\begin{example}
Following the line of the examples in the preceding section, we consider an action of the odd multiplicative supergroup $\GG_m^{1 \vert p}$ on the affine superspace $\AA^{2 \vert p+q}$, which we assume to be given by a coaction $\rho \colon k[\AA^{2 \vert p+q}] \to k[\AA^{2 \vert p+q}] \otimes k[\GG_m^{1 \vert p}]$ taken as any of those considered there. Our goal here is to understand the structure of the relative invariants with weight $\chi_1$, associated with the character $\chi_1^* \colon k[\GG_m] \to k[\GG_m^{1 \vert p}]$ which, as discussed toward the end of \S\ref{sec:2.5}, induces the identity character $\chi_{1,0}^*$ on the underlying multiplicative group $\GG_m$. In that section, it was also observed that the odd part of the Lie superalgebra of $\GG_m^{1 \vert p}$ is spanned by the odd elements $Q_1,\dots,Q_p$, which are represented as vector fields $\partial / \partial \xi_1, \dots,\partial / \partial \xi_p$. In accordance with the general framework established in \S\ref{sec:4.1}, these elements induce odd superderivations on $k[\AA^{2 \vert p+q}]$ defined by
$$
D_i = (\mathrm{id}_{k[\mathbb{A}^{2 \vert p+q}]} \otimes Q_i) \circ \rho = \left( \mathrm{id}_{k[\mathbb{A}^{2 \vert p+q}]} \otimes \frac{\partial}{\partial \xi_i} \right) \circ \rho.
$$
Evaluating the action of these superderivations under any of the coactions under consideration, we find that they admit the explicit representation
$$ 
D_i = x \frac{\partial}{\partial \eta_i}.
$$
On the other hand, recalling definition \eqref{eq:3.7}, the superalgebra of odd infinitesimal invariants is given by the common kernel of the $D_i$. Together with Proposition~\ref{prop:4.2}, this yields
$$
k[\AA^{2 \vert p+q}]^{\GG_m^{1\vert p}, \chi_1^{n}} = k[\AA^{2 \vert p+q}]^{\GG_m,\chi_{1,0}^{n}} \cap \left( \bigcap_{i=1}^{p} \ker D_i\right).
$$
Since $\chi_{1,0}$ is the identity character, the first term in this intersection is simply
$$
k[\AA^{2 \vert p+q}]^{\GG_m,\chi_{1,0}^{n}} = \{ a \in k[\AA^{2 \vert p+q}] \mid \rho_0(a) = a \otimes t^{n} \},
$$
where, as usual, $\rho_0 \colon k[\AA^{2 \vert p+q}] \to k[\AA^{2 \vert p+q}] \otimes_k k[\GG_m]$ denotes the reduced coaction. Our task is therefore to determine $\bigcap_{i=1}^{p} \ker D_i$. But this is in fact quite straightforward: for each $i = 1,\dots,p$, since $x$ is not a zero divisor in $k[\AA^{2 \vert p+q}]$, the condition $D_i a = x \frac{\partial a}{\partial \eta_i} = 0$ forces $\frac{\partial a}{\partial \eta_i} = 0$. This implies that elements in $\bigcap_{i=1}^{p} \ker D_i$ are independent of the odd variables $\eta_1,\dots,\eta_p$. Thus,
$$
\bigcap_{i=1}^{p} \ker D_i = k[x,y,\theta_1,\dots,\theta_q].
$$
Consequently, regardless of the specific coaction under consideration, the relative invariants necessarily lie within this subalgebra. Taking the direct sum over $n \ge 0$, the $\mathbb{Z}_{\ge 0}$-graded superalgebra of relative invariants decomposes as
$$
\bigoplus_{n \ge 0} k[\AA^{2 \vert p+q}]^{\GG_m^{1\vert p}, \chi_1^n} = \bigoplus_{n \ge 0} \left( k[\AA^{2 \vert p+q}]^{\GG_m, \chi_{1,0}^n} \cap k[x,y,\theta_1,\dots,\theta_q] \right),
$$
where the explicit $\mathbb{Z}_{\ge 0}$-grading is determined by the weights assigned to the generators $x, y, \theta_1, \dots, \theta_q$ under the reduced coaction $\rho_0$.
\end{example}

\begin{example}\label{ex:4.15}
Keeping in mind the characterization established above, let us now specialize the coaction to that of Example~\ref{ex:3.10}. By definition, the resulting reduced coaction $\rho_0$ is given on generators by
$$
\begin{aligned} 
 \rho_0 (x) &= x \otimes t, \\ 
 \rho_0 (y) &= y \otimes t, \\  
 \rho_0 (\eta_i) &= \eta_i \otimes t, \\ 
 \rho_0 (\theta_j) &= \theta_j \otimes t. 
 \end{aligned}
 $$
This tells us that each of the generators $x, y, \eta_1,\dots,\eta_p, \theta_1,\dots,\theta_q$ is a relative invariant under this reduced coaction, and moreover that they each carry degree $1$ in the $\ZZ_{\ge 0}$-grading. Consequently, intersecting $k[\AA^{2 \vert p+q}]^{\GG_m, \chi_{1,0}^n}$ with $k[x,y,\theta_1,\dots,\theta_q]$ yields precisely the homogeneous component $k[x,y,\theta_1,\dots,\theta_q]_n$. Summing over all $n \ge 0$, the $\ZZ_{\ge 0}$-graded superalgebra of relative invariants is
$$
\bigoplus_{n \ge 0} k[\AA^{2 \vert p+q}]^{\GG_m^{1\vert p}, \chi_1^n} = \bigoplus_{n \ge 0} k[x,y,\theta_1,\dots,\theta_q]_n = k[x,y,\theta_1,\dots,\theta_q], 
$$
where the latter now stands for the polynomial superalgebra graded by total degree in the generators $x,y,\theta_1,\dots,\theta_q$. From this, we conclude that the GIT superquotient in this case is
$$
\AA^{2 \vert p+q} {\sslash_{\chi_1}} \GG_m^{1 \vert p} = \sProj k[x,y,\theta_1,\dots,\theta_q] = \PP^{1 \vert q}.
$$
Comparing this result with Example~\ref{ex:3.2} highlights a precise super-analogue of classical GIT behavior under the scaling action. There, the affine superquotient collapses to a single elemental geometric point $\sSpec k$ because the topological closure of every geometric ray inevitably contains the origin, preventing orbit separation. Taking relative invariants under the character $\chi_1^{*}$ bypasses this collapse entirely: the character singles out functions that separate these non-zero rays, restoring the non-trivial orbit geometry and replacing the point with the projective superspace $\PP^{1 \vert q}$. To understand this structural recovery in more detail, let us examine this GIT superquotient from a local point of view.

By definition, the $\chi_1$-semistable locus $(\AA^{2 \vert p+q})^{\uss}_{\chi_1}$ is determined by the even generators of the superalgebra of relative invariants. As established above, these are precisely the even generators $x$ and $y$, each carrying degree $1$ in the $\ZZ_{\geq 0}$-grading. It follows that $(\AA^{2 \vert p+q})^{\uss}_{\chi_1}$ is covered by the basic open subsets $D(x)$ and $D(y)$, so that
$$
(\AA^{2 \vert p+q})^{\uss}_{\chi_1} = \AA^{2 \vert p+q} \setminus V(x,y),
$$
where $V(x,y)$ is the affine subsuperspace defined by the equations $x=0$ and $y=0$. Notice, in particular, that the underlying topological space of $(\AA^{2 \vert p+q})^{\uss}_{\chi_1}$ coincides with the $\chi_{1,0}$-semistable locus of the affine plane $\AA^2$ under the scaling action:
$$
\lvert (\AA^{2 \vert p+q})^{\uss}_{\chi_1} \rvert = \AA^2 \setminus \{(0,0)\}.
$$
Now, to construct the GIT superquotient $\AA^{2 \vert p+q} {\sslash_{\chi_1}} \GG_m^{1 \vert p}$, recall from Proposition~\ref{prop:4.7} that it is characterized via the GIT superquotient morphism $\pi^{\uss}_{\chi_1} \colon (\AA^{2 \vert p+q})^{\uss}_{\chi_1} \to \AA^{2 \vert p+q} {\sslash_{\chi_1}} \GG_m^{1 \vert p}$. To describe this quotient explicitly, we compute the invariants under the full coaction $\rho$ locally on the covering basic open subsets $D(x)$ and $D(y)$.  On $D(x)$, where $x$ is invertible, we introduce the local coordinates
$$
u = \frac{y}{x}, \quad \psi_j = \frac{\theta_j}{x}.
$$
Evaluating the coaction $\rho$, a direct calculation confirms that $\rho(u) = u \otimes 1$ and $\rho(\psi_j) = \psi_j \otimes 1$. Consequently, the superalgebra of invariant functions on $D(x)$ is given by
$$
k[D(x)]^{\GG_m^{1 \vert p}} = k[u,\psi_1,\dots,\psi_q],
$$
so that the open chart of the GIT superquotient corresponding to $D(x)$ is identified with the affine superspace
$$
D(x) {\sslash} \GG_m^{1 \vert p} = \sSpec k[u,\psi_1,\dots,\psi_q] \cong \AA^{1 \vert q}.
$$
In much the same way, on $D(y)$, where $y$ is invertible, introducing the local coordinates
$$
v = \frac{x}{y}, \quad \zeta_j = \frac{\theta_j}{y},
$$
yields the invariant superalgebra
$$
k[D(y)]^{\GG_m^{1 \vert p}} = k[v,\zeta_1,\dots,\zeta_q]
$$
The open chart corresponding to $D(y)$ is thus given by another copy of the affine superspace
$$
D(y) {\sslash} \GG_m^{1 \vert p} = \sSpec k[v,\zeta_1,\dots,\zeta_q] \cong \AA^{1 \vert q}.
$$
On the intersection $D(x) \cap D(y) = D(xy)$, where both $x$ and $y$ are invertible, these two open charts are linked by the transition functions
$$
v = u^{-1}, \quad \zeta_j = u^{-1} \psi_j.
$$
These are precisely the standard transition functions defining the projective superline, yielding the global isomorphism
$$
(\AA^{2 \vert p+q})^{\uss}_{\chi_1} {\sslash} \GG_m^{1 \vert p} = (\AA^{2 \vert p+q} \setminus V(x,y)) {\sslash} \GG_m^{1 \vert p} \cong \PP^{1 \vert q}.
$$
Such a description provides the desired local realization of the GIT superquotient.
\end{example}

\begin{example}
Let us now specialize the coaction to the one considered in Example~\ref{ex:3.11}. Evaluated on the generators, the resulting reduced coaction $\rho_0$ is given by
$$
\begin{aligned} 
 \rho_0 (x) &= x \otimes t, \\ 
 \rho_0 (y) &= y \otimes t^{-1}, \\  
 \rho_0 (\eta_i) &= \eta_i \otimes t, \\ 
 \rho_0 (\theta_j) &= \theta_j \otimes t. 
 \end{aligned}
$$
Thus, all generators remain relative invariants under the reduced coaction, with $x, \eta_1,\dots,\eta_p, \theta_1,\dots,\theta_q$ carrying degree $1$, while $y$ carries degree $-1$ in the induced $\ZZ_{\geq 0}$-grading. Also, as before, intersecting $k[\AA^{2 \vert p+q}]^{\GG_m, \chi_{1,0}^n}$ with $k[x,y,\theta_1,\dots,\theta_q]$ eliminates the odd coordinates $\eta_1,\dots,\eta_p$. Now, a monomial $x^i y^j \theta_{j_1} \dots \theta_{j_s}$ in $k[x,y,\theta_1,\dots,\theta_q]$ belongs to the $n$-th graded component of the superalgebra of relative invariants if and only if its total degree satisfies $i - j + s = n \geq 0$. When $i \geq j$, such a monomial factors directly into powers of $u = xy$, $x$, and the odd generators $\theta_j$. Conversely, if $i < j$, the condition $n \ge 0$ forces $s \geq j - i$, allowing us to pair $j - i$ factors of $y$ with odd variables to form the degree-zero elements $\zeta_j = y \theta_j$, while the remaining powers of $y$ collapse with $x$ into $u$. As a result, the entire $\ZZ_{\geq 0}$-graded superalgebra of relative invariants is generated by the degree-zero elements $u$ and $\zeta_j$, together with the degree-one elements $x$ and $\theta_j$, subject to the relations $x \zeta_j = u \theta_j$ for $j = 1, \dots, q$. In other words,
$$
\bigoplus_{n \geq 0} k[\AA^{2 \vert p+q}]^{\GG_m^{1 \vert p}, \chi_{1}^n} = \frac{k[u, x, \zeta_1,\dots,\zeta_q, \theta_1,\dots,\theta_q]}{(u \theta_j - x \zeta_j \mid 1 \le j \le q)}.
$$
The resulting GIT superquotient is therefore
$$
\AA^{2 \vert p+q} {\sslash_{\chi_1}} \GG_m^{1 \vert p} = \sProj \left(  \frac{k[u, x, \zeta_1,\dots,\zeta_q, \theta_1,\dots,\theta_q]}{(u \theta_j - x \zeta_j \mid 1 \le j \le q)} \right).
$$
Unlike the previous example, where the global supergeometry simplifies directly, making the precise nature of this superscheme explicit requires examining its structure locally. Let us thus consider the $\chi_1$-semistable locus $(\AA^{2 \vert p+q})^{\uss}_{\chi_1}$. Since the positive-degree component of the superalgebra of relative invariants is generated in degree $1$ by $x$ and $\theta_1, \dots, \theta_q$, this locus is determined entirely by the non-vanishing of the even generator $x$. Consequently,
$$
(\AA^{2 \vert p+q})^{\uss}_{\chi_1} = D(x),
$$
so that, on the underlying topological space, this locus coincides with the complement of the coordinate line $x=0$, namely 
$$
\lvert (\mathbb{A}^{2 \vert p+q})^{\mathrm{ss}}{\chi_1} \rvert = \mathbb{A}^2 \setminus V(x).
$$
As the semistable locus is itself the single basic open subset $D(x)$, the GIT superquotient morphism $\pi^{\uss}_{\chi_1}$ maps $D(x)$ directly onto the single basic open chart $D_+(x)$. The GIT superquotient therefore reduces to the spectrum of the degree-zero component of the localized superalgebra:
$$
\AA^{2 \vert p+q} {\sslash_{\chi_1}} \GG_m^{1 \vert p} \cong \sSpec ( k[\AA^{2 \vert p+q}]^{\GG_m^{1 \vert p}, \chi_1}[x^{-1}] )_0.
$$
To explicitly identify this degree-zero component, we introduce the coordinates
$$
u = xy, \quad \psi_j = \frac{\theta_j}{x},
$$
each carrying degree $0$ in the $\ZZ_{\geq 0}$-grading. The relation $x \zeta_j = u \theta_j$ translates directly into $\zeta_j = u \psi_j$, rendering the generators $\zeta_1, \dots, \zeta_q$ redundant. Thus, the degree-zero component is freely generated as a superalgebra by $u$ together with the odd elements $\psi_1, \dots, \psi_q$, yielding
$$
( k[\AA^{2 \vert p+q}]^{\GG_m^{1 \vert p}, \chi_1}[x^{-1}] )_0 = k[u,\psi_1,\dots,\psi_q].
$$
We thus conclude that the GIT superquotient is identified with an affine superspace
$$
\AA^{2 \vert p+q} {\sslash_{\chi_1}} \GG_m^{1 \vert p} \cong \sSpec k[u, \psi_1, \dots, \psi_q] = \mathbb{A}^{1 \vert q}.
$$
Remarkably, this GIT superquotient matches the affine superquotient computed in Example~\ref{ex:3.11}, despite the two constructions operating on distinct underlying domains. Geometrically, the character $\chi_1^*$ discards the destabilized coordinate line $x = 0$ together with the origin, restricting the action to the basic open chart $D(x)$. On this semistable locus, the local degree-zero generators $\psi_j = \theta_j / x$ and the invariant generators $\zeta_j = y \theta_j$ are linked by the relation $\zeta_j = u \psi_j$. Since the hyperbola parameter $u = xy$ is non-zero on generic closed orbits, this relation induces an isomorphism between the localized degree-zero superalgebra $k[u, \psi_1, \dots, \psi_q]$ and the invariant superalgebra $k[u, \zeta_1, \dots, \zeta_q]$. The GIT reduction thus recovers the affine superspace $\AA^{1 \vert q}$ by parameterizing the family of closed hyperbolic orbits away from the origin, preserving the odd dimension $q$ across both quotient constructions. 
\end{example}

\begin{example}
Let us finally specialize to the coaction introduced in Example~\ref{ex:3.12}. On the generators, the resulting reduced coaction $\rho_0$ acts as
$$
\begin{aligned} 
 \rho_0 (x) &= x \otimes t, \\ 
 \rho_0 (y) &= y \otimes t, \\  
 \rho_0 (\eta_i) &= \eta_i \otimes t^{-1}, \\ 
 \rho_0 (\theta_j) &= \theta_j \otimes t^{-1}. 
 \end{aligned}
$$
In particular, all generators are relative invariants, with $x, y$ carrying degree $1$ and $\eta_1,\dots,\eta_p, \theta_1,\dots,\theta_q$ carrying degree $-1$ in the induced $\ZZ_{\geq 0}$-grading. And again, when one intersects $k[\AA^{2\vert p+q}]^{\GG_m, \chi_{1,0}^n}$ with $k[x,y,\theta_1,\dots,\theta_q]$, the odd coordinates $\eta_1,\dots,\eta_p$ are eliminated. Now, a monomial $x^i y^j \theta_{j_1} \dots \theta_{j_s}$ in $k[x,y,\theta_1,\dots,\theta_q]$ lies in the $n$-th graded component of the superalgebra of relative invariants precisely when its total degree satisfies $i+j-s = n \geq 0$. This condition implies that the number of even factors $x$ and $y$ is always at least the number of odd factors $\theta_j$. As a result, each odd generator $\theta_{j_\nu}$ in the monomial can be paired with a factor of $x$ or $y$ to form the degree-zero elements $\psi_j = x \theta_j$ and $\zeta_j = y \theta_j$. The remaining unpaired factors of $x$ and $y$ then account for the positive degree $n$. Consequently, the $\mathbb{Z}_{\ge 0}$-graded superalgebra of relative invariants is generated by the degree-zero elements $\psi_1,\dots,\psi_q$ and $\zeta_1,\dots,\zeta_q$, together with the degree-one elements $x$ and $y$, subject to the defining relations $y \psi_j = x \zeta_j$ for $j = 1, \dots, q$. That is,
$$
\bigoplus_{n \geq 0} k[\AA^{2 \vert p+q}]^{\GG_m^{1 \vert p}, \chi_1^n} = \frac{k[x,y,\psi_1,\dots,\psi_q, \zeta_1,\dots,\zeta_q]}{(y \psi_j - x \zeta_j \mid 1 \le j \le q)}.
$$
Thus, the corresponding GIT superquotient takes the form
$$
\AA^{2 \vert p+q} {\sslash_{\chi_1}} \GG_m^{1 \vert p} = \sProj \left( \frac{k[x,y,\psi_1,\dots,\psi_q, \zeta_1,\dots,\zeta_q]}{(y \psi_j - x \zeta_j \mid 1 \le j \le q)} \right). 
$$
As in the previous example, the precise geometry of this superscheme is best understood by examining its local structure. Since the positive-degree component of the superalgebra of relative invariants is generated in degree $1$ by the even elements $x$ and $y$, a point fails to be $\chi_1$-semistable if and only if both $x$ and $y$ vanish. On $V(x,y)$, any non-zero monomial in $k[x,y,\theta_1,\dots,\theta_q]$ consists solely of odd factors $\theta_j$, carrying non-positive degree $-s \leq 0$. As a result, the $\chi_{1}$-semistable locus $(\AA^{2 \vert p+q})^{\uss}_{\chi_1}$ is given by the union of the two basic open subsets $D(x)$ and $D(y)$, identifying it as
$$ 
(\AA^{2 \vert p+q})^{\uss}_{\chi_1} = \AA^{2 \vert p+q} \setminus V(x,y), 
$$
whose underlying reduced scheme is the punctured affine plane,
$$
\lvert (\AA^{2 \vert p+q})^{\uss}_{\chi_1} \rvert = \AA^{2} \setminus \{(0,0)\}.
$$
In turn, the GIT superquotient morphism $\pi^{\uss}_{\chi_1}$ maps these basic open sets directly onto the basic open charts $D_+(x)$ and $D_+(y)$. Reconstructing the global GIT superquotient thus reduces to gluing these two charts across their overlap. On the chart $D_+(x)$, which is identified with the spectrum of $(k[\AA^{2 \vert p+q}]^{\GG_m^{1 \vert p}, \chi_1^n}[x^{-1}])_0$, we introduce the local coordinates
$$
u = \frac{y}{x}, \quad \psi_j = x \theta_j.
$$
In this localized superalgebra, the relation $y \psi_j = x \zeta_j$ simplifies to $\zeta_j = u \psi_j$, leaving $u$ and $\psi_1,\dots,\psi_q$ as free generators. Thus,
$$
(k[\AA^{2 \vert p+q}]^{\GG_m^{1 \vert p}, \chi_1^n}[x^{-1}])_0 = k[u,\psi_1,\dots,\psi_q]. 
$$
meaning that
$$
D_+(x) \cong \sSpec k[u,\psi_1,\dots,\psi_q] = \AA^{1 \vert q}. 
$$
Symmetrically, on the chart $D_+(y)$, identified with the spectrum of $(k[\AA^{2 \vert p+q}]^{\GG_m^{1 \vert p}, \chi_1^n}[y^{-1}])_0$, we introduce the local coordinates
$$
v = \frac{x}{y}, \quad \zeta_j = y \theta_j,
$$
under which the relation reduces to $\psi_j = v \zeta_j$, leaving $v$ and $\zeta_1,\dots,\zeta_q$ as free generators of the degree-zero component:
$$
(k[\AA^{2 \vert p+q}]^{\GG_m^{1 \vert p}, \chi_1^n}[y^{-1}])_0 = k[v,\zeta_1,\dots,\zeta_q]. 
$$
This yields
$$
D_+(y) \cong \sSpec k[v,\zeta_1,\dots,\zeta_q] = \AA^{1 \vert q}.
$$
Comparing the local coordinates of both charts, the transition functions on the overlap $D_+(x) \cap D_+(y)$ are
$$
v = u^{-1},\quad \zeta_j = u \psi_j.
$$
While the reduced scheme underlying the superscheme glued via these transition functions is the projective line $\PP^1$, these transition functions exhibit the odd coordinates as sections of the vector bundle $\Ocal_{\PP^1}(-1)^{\oplus q}$ with a degree shift. In this way, the GIT superquotient is identified as
$$
\AA^{2 \vert p+q} {\sslash_{\chi_1}} \GG_m^{1 \vert p}  \cong \Spec_{\PP^1} (\Sym_{\Ocal_{\PP^1}} \Ocal_{\PP^1}(1)^{\oplus q}) =  \tot (\Pi \Ocal_{\PP^1}(-1)^{\oplus q}),
$$
where $\Pi$ denotes the parity-shifting functor. Contrasting this GIT superquotient with that of Example~\ref{ex:4.15} highlights a fundamental geometric distinction: while that example produces the standard projective superspace $\PP^{1\vert q}$, the present one yields a ``twisted'' variant of it. This divergence stems entirely from the way the respective coactions act on the odd generators, the action on the even generators being identical in both cases.
\end{example}


\end{document}